\documentclass[11pt,reqno]{article}
\usepackage{amssymb}
\usepackage{graphicx}
\usepackage{amsmath}
\usepackage[latin1]{inputenc}
\usepackage{amssymb,indentfirst}
\usepackage{times}
\usepackage[T1]{fontenc}
\usepackage{tikz}
\newtheorem{thm}{Theorem}[section]
\newtheorem{cor}[thm]{Corollary}
\newtheorem{lem}[thm]{Lemma}
\newtheorem{defn}[thm]{Definition}
\newtheorem{rmk}[thm]{Remark}

\def\N{{\mathbb N}}
\def\Z{{\mathbb Z}}

\def\Q{{\mathbb Q}}
\def\R{{\mathbb R}}
\def\C{{\mathbb C}}
\def\SS{{\mathbb S}}

\def\bb{\begin}
\def\bc{\begin{center}}       \def\ec{\end{center}}
\def\be{\begin{equation}}     \def\ee{\end{equation}}
\def\ba{\begin{array}}        \def\ea{\end{array}}
\def\bea{\begin{eqnarray}}    \def\eea{\end{eqnarray}}
\def\beaa{\begin{eqnarray*}}  \def\eeaa{\end{eqnarray*}}
\def\bma{\begin{pmatrix}}
\def\ema{\end{pmatrix}}

\def\hh{\!\!\!\!}             \def\EM{\hh &   &\hh}
\def\EQ{\hh & = & \hh}

\def\al{\alpha}               \def\bt{\beta}
\def\e{\varepsilon}           
\def\de{\delta}               \def\om{\omega}
\def\la{\lambda}              \def\vp{\varphi}
               \def\th{\theta}
\def\vth{\vartheta}           \def\ga{\gamma}
               
               \def\Th{\Theta}
               
               \def\ro{\varrho}

\def\oo{\infty}                              
\def\dd{\cdots}                              
\def\q{\quad}                                \def\qq{\qquad}
\def\f{\frac}                                
\def\z{\left}                                \def\y{\right}
\def\tm{\times}                              \def\bs{\backslash}
\def\ol{\overline}
\def\ul{\underline}
\def\bu{$\bullet$\ }

\def\A{{\mathcal A}}

\def\mcc{{\mathcal C}}
\def\D{{\mathcal D}}
\def\E{{\mathcal E}}

\def\M{{\mathcal M}}

\def\O{{\mathcal O}}
\def\P{{\mathcal P}}

\def\mcs{{\mathcal S}}
\def\T{{\mathcal T}}

\def\rd{\,{\rm d}}
\def\dt{\,{\rm d}t}
\def\ds{\,{\rm d}s}

\def\dx{\,{\rm d}x}
\def\dy{\,{\rm d}y}

\def\ifl{\iffalse}
\def\d{\cdot}
\def\dd{\cdots}
\def\oo{\infty}
\def\f{\frac}

\def\z{\left}
\def\y{\right}
\def\q{\quad}
\def\qq{\qquad}
\def\bs{\backslash}

\def\andq{\quad \mbox{ and } \quad}
\def\qqf{\quad \forall}
\def\lb{\label}

\def\x#1{{\rm (\ref{#1})}}
\def\Proof{\noindent{\bf Proof} \quad}

\def\qed{\hfill $\Box$ \smallskip}

\def\tl{\tilde}

\def\rth{{R,\theta_0}}

\def\vr{{\vec r}}
\def\vp{{\vec p}}
\def\arctanh{{\rm arctanh}}

\def\rpq{{R_{p,q}}}
\def\Rpq{{R_{q/p}}}
\def\Rrh{R_{\ro_0}}

\def\ctper{C(\R^2/2\pi\Z^2)}

\def\ysb#1{{\color{blue} #1}}

\begin{document}

\title{On Planar Shadowing Curves to Closed Escaping Curves}

\author{Qiaoling Wei\\
School of Mathematical Sciences, Capital Normal University, Beijing 100048, China\\
E-mail: wql03@cnu.edu.cn\\
Meirong Zhang\footnote{
Supported by the National Natural Science Foundation of China (Grant no. 11790273).}\\
Department of Mathematical Sciences, Tsinghua University, Beijing 100084, China\\
E-mail: zhangmr@tsinghua.edu.cn
}


\maketitle

\begin{abstract}

We introduce a new dynamical system model called the shadowing problem, where a shadower chases after an escaper by always staring at and keeping the distance from him. When the escaper runs along a planar closed curve, we associate to the reduced shadowing equations the rotation number, and show that it depends only on the geometry of the escaping curve. Two notions called the critical shadowing distance and turning shadowing distance are introduced to characterize different dynamical behaviors. We show that a planar closed escaping curve could have shadowing curves of different types including periodic, subharmonic and ergodic ones, depending on the shadowing distance.  Singularities of cusp type are found when the shadowing distance is large. Shadowing curves to an escaping circle are examined in details analytically and numerically. Finally, we conjecture that the critical shadowing distance and turning shadowing distance are coincident for typical escaping curves.

\end{abstract}

\bigskip




\section{Shadowing Problems and Shadowing Equations}
\setcounter{equation}{0} \lb{first}

In the Euclidean space $\R^d$ of dimension $d$, suppose that a person, called the {\it escaper},  is escaping along a directional $C^1$ parameterized curve
    \[
    \E: \q \vr = \vr_0(t),\q t\in \R,
    \]
called the {\it escaping curve} (EC, for short) in this paper. Here $t$ is the time. Another person, called the {\it shadower}, is shadowing the escaper using the simplest strategy by staring at the escaper and keeping the initial distance from the escaper at all times.

Let the curves of the shadower be expressed using the following parameterization
    \[
    \mcs: \q \vr = \vr(t),\q t\in \R,
    \]
so that the starting shadowing position $\vr(0)\ne \vr_0(0)$. We can deduce the equation for $\vr(t)$ as follows. By staring at the escaper, it means that at any time $t$, one has some $\al=\al(t)\in \R$ such that
    \be \lb{eq1}
    \vr'(t) = \al(t) (\vr_0(t)-\vr(t))\qqf t\in \R.
    \ee
By keeping the initial distance, it means that
    \be \lb{eq2}
    \|\vr_0(t)-\vr(t)\|\equiv \|\vr_0(0)-\vr(0)\|\qqf t\in \R.
    \ee
By differentiating in $t$, the requirement \x{eq2} is equivalent to
    $(\vr_0-\vr)\d (\vr'_0-\vr')=0,$
i.e.
    \be \lb{rr0}
    \vr' \d (\vr_0-\vr)= \vr'_0 \d (\vr_0-\vr).
    \ee
It is also necessary from the requirement \x{eq1} that
    \bea \lb{kt}
    \al(t) \EQ \f{\vr'(t)\d (\vr_0(t)-\vr(t))}{\|\vr_0(t)-\vr(t)\|^2} \equiv \f{\vr'_0(t)\d (\vr_0(t)-\vr(t))} {\|\vr_0(t)-\vr(t)\|^2},
    \eea
when equality \x{rr0} is used. By substituting \x{kt} into the requirement \x{eq1}, we conclude that the motions $\vr(t)$ of the shadower must satisfy
    \be \lb{SE}
    \vr'  =\f{\vr'_0(t)\d (\vr-\vr_0(t))}{\|\vr-\vr_0(t)\|^2}(\vr-\vr_0(t))=:\vec f(t,\vr).
    \ee
This is a well-defined nonlinear non-autonomous system of ODEs in $\R^d$ with the time dependent vector field $\vec f(t,\vr)$. Conversely, it is easy to verify that any solution $\vr(t)$ of system \x{SE} fulfills the requirements  \x{eq1} and \x{eq2}, with $\al(t)$ as in \x{kt}.

Geometrically, system \x{SE} means that $\vr'$ is the projection of $\vr_0'$ in the direction $\vr-\vr_0$. The shadower $\vr(t)$ will point to (resp. oppose to) the escaper $\vr_0(t)$ when $\alpha(t)>0$ (resp. $\alpha(t)<0$). When $\alpha(t)=0$ or $\vr_0'(t) \cdot(\vr(t)-\vr_0(t))=0$, the shadower $\vr(t)$ will stop at these times.

With these explanations to the strategy that the shadower is always staring at the escaper, system \x{SE} of ODEs is called in this paper the {\it shadowing equation} (SE, for short) to (the EC) $\E$, and, meanwhile, the solutions of SE \x{SE} are called the {\it shadowing curves} (SC or SCs, for short) to $\E$ or to $\vr_0(t)$. Moreover, for any SC $\vr(t)$ to $\vr_0(t)$,
    \be \lb{sdr}
    R:= \|\vr_0(0)-\vr(0)\|\equiv \|\vr_0(t)-\vr(t)\|>0
    \ee
is called the {\it shadowing distance} 
of SC $\vr(t)$.

    \bb{thm}\lb{sys}
Let the EC $\E$ be given. Then, for any initial point $\vr(0)$ different from $\vr_0(0)$, SE \x{SE} admits a unique globally defined solution
    \be \lb{vrtt}
    \vr=\vr(t)
    =\vr(t;\vr(0)), \qq t\in \R.
    \ee
Hence any solution \x{vrtt} defines a parameterized SC to $\vr_0(t)$ with the shadowing distance $R$ being defined by \x{sdr}.
    \end{thm}

To see that solutions \x{vrtt} are globally defined, one can notice from \x{sdr} that solutions $\vr(t)$ are always bounded away from singularities $\vr=\vr_0(t)$. Moreover, the vector field $\vec f(t,\vr)$ of SE \x{SE} is smooth in $\vr$ and satisfies the boundedness condition
    \[
    \|\vr'(t)\|=\|\vec f(t,\vr)\|\le \|\vr'_0(t)\|.
    \]
Hence solutions $\vr(t)=\vr(t;\vr(0))$ can always be continued to the whole line of $t$.

Besides the notations \x{vrtt} for SCs, we also use the following notations for SCs
    \[
    \mcs=\mcs_{\vr(0)}= \mcs_{\vr(0);\vr_0}, \mbox{ etc.}
    \]
Here, in the last notation, the dependence of SCs on the EC $\vr_0(\d)$ is emphasized.
\ifl When time $t$ is taken in account, these are also written as
    \[
    \mcs(t)=\mcs_{\vr(0)}(t)= \mcs_{\vr(0);\vr_0}(t), \mbox{ etc.}
    \]
\fi

In this paper, we mainly concentrate on the studying for shadowing curves on the Euclidean plane $\R^2$ when the escaping curves are planar closed curves. The content and results are as follows.

In \S \ref{second}, we will first briefly study the invariance properties on shadowing problems. Then we will deduce an extending shadowing equation \x{ESE} for general dimension which is a higher dimensional linear system. Finally, when the planar shadowing curves to planar escaping curves are considered, we will use the moving polar coordinates to deduce a reduced shadowing equation \x{RSE} which is a nonlinear differential equation on the circle.

The main content is given in \S \ref{third}. When the escaping curve is a planar closed curve $\E$ with some regularity, we use the reduced shadowing equation to introduce the rotation number $\ro(R)=\ro_\E(R)$ from dynamical systems theory \cite{A83, H80, KH95}, which is a function of the shadowing distance $R\in(0,+\oo)$.
It is proved in Lemma \ref{inv} that $\rho_\E(R)$ is independent of the parameterizations of $\E$, i.e. $\rho_\E(R)$ depends only on the geometry of $\E$. Moreover, we find that $\ro_\E(R)$ has closed connections with the perimeter of $\E$, the rotation index of  $\E$, and the area enclosed by $\E$.
For details, see Lemma \ref{upbound}, Theorem \ref{rho-01} and Theorem \ref{mono}. By using rotation number $\ro(R)$, we will apply the dynamical behavior of circle diffeomorphisms,  including the Denjoy theorem, to give a fair complete characterization of types of planar shadowing curves.
 The main results are stated in Theorem \ref{main1} and Theorem \ref{main2}. Typically, we have the following three types of shadowing curves:
\begin{itemize}
\item when the shadowing distance $R$ is not too large, the shadowing problem admits only $2\pi$-periodic shadowing curves and those shadowing curves which are approaching to periodic ones.
\item when $R$ is large enough and $\ro(R)$ is rational, the shadowing problem admits subharmonic ($2p\pi$-periodic) shadowing curves and those shadowing curves which are approaching to subharmonic ones.
\item  when $R$ is large enough and $\ro(R)$ is irrational, each shadowing curve is dense in the shadowing domain $\D_R$ as in \x{SDD}.
\end{itemize}

In order to distinguish the ranges of these different shadowing distances, we use the properties of rotation number $\ro_\E(R)$ to introduce two notions which are called the critical shadowing distance $\ul{R}(\E)$ and the turning shadowing distance $\ol{R}(\E)$. See Definition \ref{CSD} and Definition \ref{TSD} respectively. These notions depend only on the geometry of closed escaping curves $\E$.
It is proved in Theorem \ref{circ} that $\ul{R}(\E)$ is really different from the `perimeter' when $\E$ is not a circle. Finally, regularity of shadowing curve are shown for convex escaping curve in section \ref{scs2}.

In \S \ref{fourth}, by considering the unit circle as an escaping curve, we will examine all shadowing curves in details using
the reduced shadowing equation. For this simplest example, the shadowing problem will admit beautiful shadowing curves. These will be plotted in Figures \ref{sccircles}-\ref{Erg4}. When the escaping curve is chosen an ellipse, we give some analytic and numerical analysis to possible shadowing curves.

In \S \ref{fifth}, we impose a conjecture, which asserts that the critical shadowing distance $\ul{R}(\E)$ and the turning shadowing distance $\ol{R}(\E)$ are coincident: $\ul{R}(\E)=\ol{R}(\E)$ for typical closed curves $\E$. Once this is true, we can give a compete characterization to all types of shadowing curves. As mentioned in \S \ref{third}, these quantities are related with the geometric properties of $\E$. Hence the conjecture may be of independent interest from the point of view of differentiable geometry.

Finally, although the paper contains several interesting results, it is just a beginning study for shadowing problems. Moreover, most of the proofs in this paper are not difficult from the point of view of dynamical systems.

\section{Extended and Reduced Shadowing Equations}
\setcounter{equation}{0} \lb{second}

\subsection{Invariance properties on SCs} \lb{s21}

Let us state some invariance results on SCs under the temporal and spatial transformations. At first, we consider temporal transformations. Let $\T: \R\to \R$ be a $C^1$ diffeomorphism, regardless increasing or decreasing, considered as a temporal transformation. Then $\T$ transforms any $C^1$ curve $\vec \ga(t)\in C^1(\R, \R^d)$ to another $C^1$ curve $\vec \ga(\T(t))\in C^1(\R, \R^d)$.  From the SEs and the definition of SCs, it is trivial that
    \[
    \mbox{$\vr(\T(t))$ is an SC to $\vr_0(\T(t))$ $\iff$ $\vr(t)$ is an SC to $\vr_0(t)$ .}
    \]
Moreover, under $\T$, the initial escaping and shadowing points of $\vr_0(t)$ and $\vr(t)$ are transformed to $\vr_0(\T(0))$ and $\vr(\T(0))$  respectively. Hence the shadowing distances are invariant under $\T$. In particular, one can consider linear temporal transformations defined by
    \[
    \T_{\al,\bt}(t) := \al t+\bt, \qq t\in \R,
    \]
where $\al, \ \bt \in \R$ are constants such that $\al\ne 0$.

Next we consider spatial transformations $\M$ of $\R^d$ defined by translations, rigid rotations and dilations. More precisely, let $\vec b\in \R^d$, $\O\in O(\R^d)$ and $D>0$, one has a transformation $\M: \R^d \to \R^d$ defined by
    \[
    \M(\vec x) := D \O(\vec x)+\vec b,\qq \vec x\in \R^d.
    \]
Then
    \[
    \mbox{$\M(\vr(t))$ is an SC to $\M(\vr_0(t))$ $\iff$ $\vr(t)$ is an SC to $\vr_0(t)$.}
    \]
Under $\M$, the initial escaping and shadowing points of $\vr_0(t)$ and $\vr(t)$ are  transformed to $\M(\vr_0(0))$ and $\M(\vr(0))$ respectively. Hence the shadowing distances are dilated
    \[
    \z\|\M(\vr_0(t))-\M(\vr(t))\|\equiv D \|\vr_0(t)-\vr(t)\y\|.
    \]

\subsection{Extended shadowing equations} \lb{ese9}

In the following, we will show that the shadowing problems in $\R^d$ can be reduced from some linear systems of ODEs, but in a higher dimensional Euclidean space $\R^{d+1}$. In $\R^d$, the $(d-1)$-dimensional unit sphere is denoted by
    \[
    \SS^{d-1}:= \bigl\{ \vec x\in \R^d: \|\vec x\|=1\bigr\}.
    \]

When an EC $\vr_0(t)$ is given, let us introduce a family of linear systems of ODEs in $\R^{d+1}$. Let $R>0$ be any number, considered as a parameter here. The linear system for $(\vec x,y) \in\R^d\tm \R=\R^{d+1}$ is defined to be
    \be \lb{ESE}
    \z\{\ba{l}\vec x'(t)= -\f{1}{R}\vr'_0(t)y(t),\\
    y'(t) = -\f{1}{R}\vr'_0(t) \d\vec x(t) .
    \ea\y.
    \ee

    \bb{lem}\lb{fir}
For any $R>0$, the function $y^2-\|\vec x\|^2$ is a first integral of system \x{ESE}. Hence solutions of  \x{ESE} are foliated as
    \be \lb{Hc}
    {\mathcal H}_c: \q y^2 - \|\vec x\|^2 \equiv c\in \R,
    \ee
a family of invariant hyperbolas.
    \end{lem}

\Proof
Let $(\vec x(t),y(t))$ be any solution of \x{ESE}. Then
    \beaa
    \EM \f{\rd}{\dt} \z((y(t))^2 -\|\vec x(t)\|^2\y) = 2\z( y(t) y'(t)-\vec x(t)\d \vec x'(t)\y)\\
    \EQ -\f{2}{R} \z(y(t) \vr'_0(t) \d\vec x(t)- \vec x(t)\d \vr'_0(t)y(t)\y)\equiv 0.
    \eeaa
Hence $y^2-\|\vec x\|^2$ is a first integral of system \x{ESE}.
\qed

When $c=0$, \x{Hc} is
    \be \lb{H0}
    {\mathcal H}_0: \q y^2 - \|\vec x\|^2 \equiv 0,
    \ee
an invariant cone of \x{ESE}.

    \bb{thm} \lb{ESE1}
For any initial shadowing point $\vr(0)\ne \vr_0(0)$, let us take in \x{ESE} the parameter $R=\|\vr(0)-\vr_0(0)\|$. Suppose that $(\vec x(t),y(t))$ is a solution of \x{ESE} satisfying the initial value conditions
    \be \lb{x0y0}
    (\vec x(0),y(0))=\z((\vr(0)-\vr_0(0))/R,1\y)\in\SS^{d-1}\tm \{1\}.
    \ee
Then
    \be \lb{ybg0}
    y(t)> 0\qqf t\in \R,
    \ee
and the SC $\vr(t)=\vr(t;\vr(0))$ to the EC $\vr_0(t)$ is given by
    \be \lb{rxy}
    \vr(t)\equiv \vr_0(t) + R \f{\vec x(t)}{y(t)}, \qq t\in \R.
    \ee
    \end{thm}

\Proof
Due to the initial value conditions \x{x0y0}, we know from \x{H0} that $(\vec x(t),y(t))$ satisfies
    \be \lb{h00}
    (y(t))^2 - \|\vec x(t)\|^2 \equiv 0.
    \ee
If \x{ybg0} fails, we have $y(t_0)=0$ for some $t_0$. Combining with \x{h00}, we have also $\vec x(t_0) =\vec 0$. As $(\vec x(t),y(t))$ is a solution of a linear system, we would have $(\vec x(t),y(t))\equiv (\vec 0,0)$, which contradicts \x{x0y0}.

Because of \x{ybg0}, the right-hand side of \x{rxy} is well-defined and is temporarily denoted by
    \[
    \vec u(t):=\vr_0(t) + R \f{\vec x(t)}{y(t)}.
    \]
Then, by using system by \x{ESE}, the derivative of $\vec u(t)$ is
    \beaa
    \vec u' = \vr'_0 + \f{R \vec x' y - R y' \vec x }{y^2}= \vr'_0 - \f{\vr'_0 y y - (\vr'_0\d \vec x) \vec x }{y^2}
    = \f{1}{y^2} (\vr'_0 \d \vec x) \vec x.
    \eeaa
On the other hand, one has from \x{h00} that $\vec u-\vr_0= R {\vec x}/{y}\in R \SS^{d-1}$. Hence
    \beaa
    \vec f(t,\vec u)= \f{\vr'_0 \d (\vec u-\vr_0)}{\|\vec u-\vr_0\|^2}(\vec u-\vr_0)= \f{\vr'_0 \d (R {\vec x}/{y})}{R^2}(R {\vec x}/{y})= \f{1}{y^2} (\vr'_0 \d \vec x) \vec x.
    \eeaa
That is, $\vec u(t)$ also solves SE \x{SE}. By \x{x0y0}, one has $\vec u(0)= \vr(0)$. Thus $\vec u(t)\equiv \vr(t)$, completing the proof of equality \x{rxy}.
\qed

Because of Theorem \ref{ESE1}, system \x{ESE} is referred to the {\it extended shadowing equation} (ESE, for short) to $\vr_0(t)$.

    \bb{rmk}\lb{ft}
When the shadowing distance $R$ is fixed, all SCs can be obtained from solutions of linear system \x{ESE}. In such a sense, shadowing problems are simple dynamical systems. For example, suppose that $\E$ is closed, i.e. $\vr_0(t)$ is periodic. It can be expected from the Floquet theory \cite{H80} for time-periodic linear systems that shadowing problems can admit periodic and quasi-periodic SCs. On the other hand, as ESEs \x{ESE} depend on the parameter $R$, shadowing problems have a relatively abundant structure of SCs by changing shadowing distance $R$. In fact, in this paper, we will excavate the main features on the structure and types of SCs by taking in account of such a parameter.
    \end{rmk}

\subsection{Reduced shadowing equations} \lb{rse}

We consider a general $C^1$ directional planar parameterized escaping curve
    \[
    \E: \q \vr_0(t) = (\xi(t),\eta(t))\in C^1(\R, \R^2).
    \]
With a given shadowing distance $R>0$, let us adopt the moving polar coordinates
    \be \lb{mpc}
    \mcs:\q \vr(t)= \vr_0(t)+ (R\cos \th(t), R \sin \th(t)).
    \ee
From \x{mpc}, one has
    \[
    \vr'= \vr'_0+  R \th'(-\sin \th, \cos \th).
    \]
Then SE \x{SE} is
    \[
    \vr'_0+  R\th'(-\sin \th, \cos \th)=  \z(\vr'_0\d(\cos \th, \sin \th)\y)(\cos \th, \sin \th).
    \]
By taking the inner product with $(-\sin \th, \cos \th)$, we conclude that $\th=\th(t)$ is determined by the scalar ODE
    \be \lb{RSE}
    \th'   = -\f{1}{R}(-\xi'(t)\sin \th +\eta'(t)\cos \th):=F_{\vr_0}(t,\th),
    \ee
with the initial condition
    \be\label{th0}
\theta(0)=\theta_0,\quad \text{where}\,(\cos \th_0,\sin \th_0)=(\vr(0)-\vr_0(0))/R.
    \ee
We call Eq. \x{RSE} the {\it reduced shadowing equation} (RSE, for short) to the EC $\vr_0(t)$. SCs are also denoted as $\mcs_\rth$ or $\vr_\rth(t)$. 

In addition, we shall call 
\be\lb{SDD}\mathcal{D}_{R}:=\{\vr_0(t)+Rz|t\in \R, z\in \R^2, |z|=1\}\ee
the {\it Shadowing domain} with given shadowing distance $R$.

\section{Rotation Numbers and Types of Shadowing Curves}
\setcounter{equation}{0} \lb{third}

In this section, we  will consider a general $C^1$ closed escaping curve $\E$. Up to a temporal change, we can always parameterize $\E$ with minimal period $2\pi$, i.e.
    \be\label{pa1}
    \E: \vr_0(t)=(\xi(t),\eta(t))\in C^1(\R,\R^2),\quad (\xi(t+2\pi),\eta(t+2\pi))=(\xi(t),\eta(t)).
    \ee
    In the sequel, we always assume $\E$ is regular, i.e.     \be \lb{reg}
    \vr'_0(t)\ne \vec 0\q\forall t\in\R.
    \ee

\subsection{Rotation numbers} \lb{rn}

Given a closed EC $\E$ as in \x{pa1}, observe that  the corresponding RSE  \x{RSE} is $2\pi$-periodic in $t$ and in $\th$, 
 we will associate to $\E$ the rotation number to characterize the dynamics of SCs.

    \begin{thm} \lb{rot-F} ({\cite{H80} Theorem 2.1})
Suppose that $F(t,\theta)\in C(\R^2/2\pi \Z^2)$  and $F(t,\th)$ is $C^1$ in $\theta$. Let $\theta(t;\theta_0)$ denote the solution of the following equation
    \be\label{per}
    \theta'=F(t,\theta),\quad \theta(0)=\theta_0,
    \ee
then the rotation number of the equation
    \be \lb{rhoF}
    \ro=\ro(F):=\lim_{|t|\to \infty}\frac{\theta(t;\theta_0)-\theta_0}{t}\in \R
    \ee
exists and is independent of $\theta_0$. Moreover, the map
    \[
    \ro: C(\R^2/2\pi \Z^2)\to \R,\quad F\mapsto \ro(F)
    \]
is continuous with respect to the $C^0$ norm in $C(\R^2/2\pi\Z^2)$.
    \end{thm}

\begin{rmk} \lb{rmk22}
Rotation number is an important tool in lower dimensional dynamical systems \cite{A83, H80, KH95} and has many applications in different problems \cite{GZ00, JM82}.
In an abstract setting, the Poincar\'e map $\P: \theta_0\mapsto \theta(2\pi;\th_0)$ of Eq. \x{per} defines a monotone homeomorphism on the circle $\R/{2\pi \Z}$ with the rotation number being defined as
    \be\label{rot2}
    \ro_\P:=\lim_{n\to \infty}\frac{1}{2\pi}\frac{\P^n(\th_0)-\th_0}{n}.
    \ee
Since $\th(t,\th(2n\pi,\th_0))=\th(t+2n\pi,\th_0)$ by periodicity of $F$, the rotation number $\ro(F)$ defined in \x{rhoF} equals the rotation number $\ro_\P$ of $\P$ as in \x{rot2}. We remark that since $\th(t;\th_0)$ is itself defined on $\R$, the rotation number does not need to modulo $\Z$.

\end{rmk}

The dynamical behavior revealed by rotation number is summarized as follows, which is standard in one dimensional circle dynamics  theory.

\begin{thm} ({\cite{H80} Theorem 2.2 \& Theorem 2.4})\label{circledyn}
\begin{enumerate}
\item If the rotation number $\ro=\ro(F)$ is rational, then Eq. \x{per} admits closed trajectories and every other trajectory approaches a closed one.
\item  (Denjoy) If $\ro=\ro(F)$ is irrational and the Poincar\'e map $\mathcal{P}$ is $C^2$, then $\mathcal{P}$ is minimal (i.e. every orbit is dense).
\end{enumerate}
\end{thm}

Given a $C^1$ closed EC $\E$ as in \x{pa1}, the rotation number of RSE \x{RSE} is denoted by $\ro(R)$, which is considered as a function of shadowing distances $R\in(0,\oo)$.  From Theorem \ref{rot-F}, we know that $\ro(R)$ is continuous in $R\in (0,\oo)$.

\smallskip

The following lemma shows that the rotation number is independent of the parameterizations of $\E$. We use the notation $\ro_{\E}(R)$ to emphasize that $\ro(R)$ is determined by $\E$.

\begin{lem} \lb{inv}
Let $\E$ be a $C^1$ regular closed curve as in \x{pa1} and \x{reg}. Then for any fixed shadowing distance $R>0$, the rotation number $\ro(R)=\ro_{\E}(R)$ is independent of the parameterizations of $\E$.
\end{lem}

\Proof
Let $\E$ be as in \x{pa1} and \x{reg}, which are parameterized using $t$. The arc-length parameter $\tilde{s}$ is given by
    \[
    \tilde{s}=\int_0^t B(t')\dt'=:\alpha(t),\quad  B(t):=\|\vr_0'(t)\|>0.
    \]
It follows that
    \[
    \alpha(t+2\pi)=\alpha(t)+\int_0^{2\pi} B(t)\dt=\alpha(t)+\ell_0,
    \]
where $\ell_0$ is the perimeter of $\E$. Let $\beta(\tilde{s}):=\alpha^{-1}(\tilde{s})$. Then
    \be \lb{para0}
    \beta(\tilde{s}+\ell_0)=\beta(\tilde{s})+2\pi,\quad \beta(\ell_0)=2\pi.
    \ee
Define
    \be \lb{para}
    \mu:=\frac{\ell_0}{2\pi}\andq s=\mu^{-1}\tilde{s}.
    \ee
Then the normalized\footnote{The term ``normalized'' means normalization of parameter such that the period changes from $\ell_0$ to $2\pi$.} arc-length parametrization of $\E$
    \[ \vr_*(s):=\vr_0(t)=\vr_0(\beta(\tilde{s}))=\vr_0(\beta(\mu s))\]
is $2\pi$-periodic in $s$. For fixed shadowing distance $R$, if $\th(t)$ is a solution of the RSE $\th'(t)=F_{\vr_0}(t,\th)$,  then $\th_*(s):=\th(\beta(\mu s))$ is  a solution of $\th_*'(s)=F_{\vr_*}(s,\th_*)$. Thus
    \[
    \ro(F_{\vr_*})=\lim_{s\to \infty}\frac{\th_*(s)}{s}=\lim_{s\to\infty}\frac{\th(\beta(\mu s))}{\beta(\mu s)}\lim_{s\to\infty}\frac{\beta(\mu s)}{s}=\lim_{t\to\infty}\frac{\th(t)}{t}=\ro(F_{\vr_0}),
    \]
because we have from \x{para0} and \x{para} the limit
    \( \lim_{s\to\infty}\frac{\beta(\mu s)}{s}=1. \)
\qed


\begin{cor}
\begin{enumerate}
\item In the normalized arc-length parametrization $\vr_*(s)$ of $\E$, the escaper is running at the constant speed $\mu:=\frac{\ell_0}{2\pi}$.
\item Given any $c>0$, let $\E_{c}:=c\E$, then $\ro_{\E_{c}}(R)= \ro_{\E}(c^{-1} R)$.
\end{enumerate}
\end{cor}

\begin{proof} It is easy to calculate
    \[\|\vr'_*(s)\|=\|\vr_0'(\beta(\mu s))\|\beta'(\mu s)\mu=\alpha'(\beta(\mu s))\beta'(\mu s)\mu=\mu.\]
where $\alpha$, $\beta$ are as in Lemma \ref{inv}.  The second statement follows directly from the form of RSE \x{RSE}.\qed
\end{proof}

\begin{rmk} \lb{rmk33}
Up to translations,  the arc-length parametrization of $\E$ can be uniquely given by
    \[\vr_0(\tilde{s})=\z(\int_0^{\tilde{s}} \cos \varphi(s')\ds',\int_0^{\tilde{s}}\sin \varphi( s')\ds'\y),\]
where $\varphi(\tilde{s})$ is the angle between the $x$-axis and the tangent vector of $\E$ at $\tilde{s}$, whence the signed curvature function is $\kappa(\tilde{s})=\varphi'(\tilde{s})$.
The normalized arc-length parameterization of $\E$ is then $\vr_*(s)=\vr_0(\mu s)$ and 
    \be\label{arclength}\vr'_*(s)=\mu \big(\cos\psi_*( s),\sin\psi_*(s)\big),\quad \psi_*(s):=\varphi(\mu s).\ee
Remark that the geometry of a closed curve is determined by its signed curvature. See \cite{DC57} for reference.
\end{rmk}

\subsection{Critical shadowing distances} \lb{csd}

In the sequel, we consider  $C^1$ regular closed EC $\E$ as in \x{pa1} and \x{reg} with some given parameterization $\vr_0(t)$.
Write $\vr_0'(t)$ in polar coordinates as
    \be \lb{ec23}
    \vr'_0(t) \equiv B(t) \z( \cos \psi(t), \sin \psi(t)\y),
    \ee
where $B(t)=\|\vr_0'(t)\|>0$, and $\psi(t)$ is a continuous function which is uniquely determined by choosing
    \(
    \psi(0) \in[-\pi,\pi).
    \)
Moreover, due to the $2\pi$-periodicity of $\vr'_0(t)$, one has
    \be \lb{Avp}
    B(t+2\pi) \equiv B(t), \andq \psi(t+2\pi) \equiv 2\pi \om_0 + \psi(t),
    \ee
where
    \[
    \om_0= \om(\E):=\frac{1}{2\pi}(\psi(2\pi)-\psi(0))\in \Z
    \]
is the {\it rotation index} of  $\E$ (\cite{DC57}). It measures the complete turns given by the tangent vector field along a closed curve. Equivalently, rotation index $\om(\E)$ is the same as the winding number of its tangential curve $\E':\R\to \R^2, t\mapsto \vr_0'(t)$.

    \bb{rmk} \lb{C2}
Suppose further that $\E$ is $C^2$. Then the following two geometric invariants can be given explicitly:
    \begin{enumerate}
    \item  The rotation index $\omega$ of $\E$ is
    \[
    \om_0=\frac{1}{2\pi}\int_0^{2\pi}\psi'(t)\dt= \f{1}{2\pi}\oint_{\E'} \f{x\dy- y\dx}{x^2+y^2}.
    \]
    \item The signed curvature $\kappa$ of $\E$ is
    \be\label{curvature}\kappa(t)=\frac{det(\vr_0'(t),\vr_0''(t))}{\|\vr_0'(t)\|^3}=\frac{\psi'(t)}{B(t)}.\ee
    \end{enumerate}

    \end{rmk}

    \begin{lem} \label{upbound}
The function $\ro: R\mapsto \ro(R)$ is continuous  and has the following estimation
    \[
    |\ro(R)|\leq \frac{\ell_0}{2\pi R}\quad \forall R\in (0,\infty),
    \]
where $\ell_0=\ell(\E):=\int_0^{2\pi }\|\vr_0'(t)\|\dt$ is the algebraic perimeter of $\E$.
    \end{lem}

\Proof  One has from Eq. \x{RSE} that
    \[
    |\th'(t)| \le \f{1}{R} {\|\vr'_0(t)\|}.
    \]
Hence
    \[
    |\th(t)|\le |\th_0| +  \f{1}{R} \int_0^t \|\vr'_0(s)\| \ds\qqf t\ge 0,
    \]
and
   \[
   |\ro(R)| \le \lim_{t\to+\oo} \f{1}{t}\z(|\th_0| +  \f{1}{R} \int_0^t \|\vr'_0(s)\| \ds\y) = \f{1}{2\pi R} \int_0^{2\pi} \|\vr'_0(s)\|\ds = \f{\ell_0}{2\pi R}.
    \]
\qed

Now we give a key lemma for the understanding of rotation numbers with small shadowing distances.

\begin{lem}\label{existenceper}
Consider the periodic ODE \x{per}. Assume that
    \be\label{con1} F(t,-\pi)>0,\qq F(t,0)<0,\quad \forall t\in \R.\ee
Then Eq. \x{per} admits two $2\pi$-periodic solutions such that $\th_+(t)\in (-\pi,0)$ and $\th_-(t)\in (0,\pi)$ for all $t$.
\end{lem}

\Proof
Consider the associated Poincar\'e map $\P(\th_0):=\th(2\pi;\th_0)$. It follows from \x{con1} that $\P([-\pi, 0])\subset (-\pi ,0)$. Hence $\P$ admits a fixed point $\th^+\in (-\pi,0)$. The solution $\th_+(t):=\th(t;\th^+)$ is the desired $2\pi$-periodic solution. The existence of $\th_-(t)$ comes from the observation that $F(t,0)<0$ and $F(t,\pi)>0$, and then by simply reversing the time, one has $\P^{-1}([0,\pi])\subset (0,\pi)$. \qed

    \bb{thm} \lb{rho-01}
Let $\E$ be $C^1$ regular closed EC. Then there exists $\tl R>0$ such that
    \[
    0<R \le \tl R \q \Longrightarrow \q \ro(R) \equiv \om_0=\om(\E).
    \]
Moreover, for each $R\in (0,\tilde{R})$, RSE \x{RSE} admits at least two $2\pi$-periodic solutions.
    \end{thm}

\Proof Substitution of \x{ec23} to  RSE \x{RSE} leads to
    \be \lb{rse3}
    \th'  = - \f{1}{R}B(t) \cos \z(\th-(\psi(t)-\pi/2) \y)\in \ctper.
    \ee

Suppose first that $\psi(t)$ is $C^1$.  Let
    \be \lb{ch1}
    \phi:=\th-(\psi(t)-\pi/2)
    \ee
in \x{rse3}. Then $\phi=\phi(t)$ satisfies
    \be\label{rse40}
    \phi'=-\frac 1 R B(t)\cos \phi-\psi'(t)=:F_R(t,\phi).
    \ee
Take
    \be\label{radii1}
    \tilde{R}=\min_{t\in [0,2\pi] }\frac{B(t)}{|\psi'(t)|}=:R_{min}>0.
    \ee
Then for $R\in (0,\tilde{R})$,  $F_R$ satisfies condition \x{con1}. Hence Eq. \x{rse40} admits two $2\pi$-periodic solutions 
\be\label{persol1}\phi_R(t;\phi_R^{\pm})\subset (\mp\pi ,0),\quad t\in \R\ee

 Choose $\th_R^{\pm}:=\phi_R^{\pm}+\psi(0)-\pi/2$.
Then
    \be\label{persol}
    \theta_R(t;\th_R^{\pm})\equiv \psi(t) - \f{\pi}{2} + \phi_R(t;\phi_R^{\pm})
    \ee
are solutions of Eq. \x{rse3}. Therefore
    \[
    \ro(R)=\lim_{t\to \infty} \frac{\th_R(t;\th_R^{\pm})}{t}=\lim_{t\to\infty}\frac{\psi(t)}{t}=\omega_0,\quad\forall R\in (0,\tilde{R}).
    \]
See the second equality of \x{Avp} and the definition of rotation index $\om_0$.

In general, $\psi(t)$ is only $C^0$. By the Weierstrass theorem, we can pick up any $\tilde{\psi} (t)\in C^1(\R/{2\pi \Z})$ such that
    \[\|\psi-\tilde{\psi}\|=\max_{t\in [0,2\pi]}|\psi(t)-\tilde{\psi}(t)|<\frac {\pi}{2}.\]
Now let us modify the transformation \x{ch1} as
    \[
    \phi:=\th-(\tilde{\psi}(t)-\pi/2).
    \]
Then Eq. \x{rse3} becomes

    \[\phi'=-\frac 1 R B(t)\cos (\phi+\tilde{\psi}(t)-\psi(t))-\tilde{\psi}'(t)=:\tilde{F}_R(t,\phi).\]
Take
    \be\label{radii2}\tilde {R}=\min_{t\in [0,2\pi]}\frac{B(t)\cos (\tilde{\psi}(t)-\psi(t))}{|\tilde{\psi}'(t)|}>0.\ee
Then for $R\in (0,\tilde{R})$,  $\tilde{F}_R$ still satisfies condition \x{con1}. Applying similar arguments as before, we conclude that $\ro(R)=\omega_0$ for $R\in (0,\tilde {R})$.
\qed

By Remark \ref{C2}, the quantity $R_{min}$ in \x{radii1} is a geometric invariant which is just the minimum of the radius of curvature of $\E$.

    \begin{rmk} \lb{better}
     It is possible that a well-chosen $\tilde{\psi}$ could make the estimation of $\tilde{R}$ in \x{radii2} be better than $R_{min}$.
     \end{rmk}

The following result is concerned with the rotation number for large shadowing distances.

    \bb{thm} \lb{mono}
Assume that $\E$ is a $C^1$ closed EC. Then 
    \be \lb{rho-ra}
    \ro(R) = \f{\A_0}{2\pi R^2} +o\z(\f{1}{R^2}\y)\qq \mbox{as } R\to +\oo,
    \ee
where
    \[
    \mathcal{A}_0= \A(\E):= \oint_\E x\dy
    = - \oint_\E y\dx =  \f{1}{2}\oint_\E x\dy-y\dx
    \]
is the algebraic area enclosed by $\E$.  In particular, $\ro(R)$ is non-increasing (resp. non-decreasing) if $\mathcal{A}_0>0$ (resp. $\mathcal{A}_0<0$) for $R$ large enough.%
    \end{thm}

\Proof Set $\lambda:=1/R$. Then Eq. \x{RSE} writes as
    \be \lb{rse4}
    \th'=  \la \vr'_0(t)\d (\sin \th, -\cos \th),\quad {where}\, \,\vr_0(t)=(\xi(t),\eta(t)).
    \ee
Here we extend the parameter $\la$ to $\R$.  The solution $\th(t;\th_0,\la)$ is analytic in $(\th_0,\la)$. By \x{rse4}, at $\la=0$, one has
    \[
    \th(t;\th_0,0) \equiv \th_0.
    \]
Consider the power series of  $\th(t;\th_0,\la)$ of $\lambda$ at $\lambda=0$,
    \be \lb{Expand}
    \th(t;\th_0,\la) = \th_0 + \sum_{k= 1}^\oo\th_k(t)\lambda^k,
    \ee
where $\th_k(t)=\th_k(t;\th_0)$. Inserting expansions \x{Expand}  into Eq. \x{rse4}, we obtain
    \beaa
    \th'_1(t) \EQ \vr'_0(t)\d (\sin \th_0, -\cos \th_0), \\
    \th'_2(t) \EQ \z(\vr'_0(t)\d (\cos \th_0, \sin \th_0)\y) \th_1(t),
    \eeaa
with initial conditions $\th_{1}(0)=\th_2(0)=0$. One has
    \[
    \th_1(t) = (\vr_0(t)-\vr_0(0))\d (\sin \th_0, -\cos \th_0).
    \]
    In particular, $\th_1(2\pi)=0$.
Inserting the solution $\th_1(t)$ into the equation of $\th_2'$ and integrate over $[0,2\pi]$, we get
 \beaa
 \th_2(2\pi)\EQ \int_0^{2\pi}\th_2'(t)\dt\\
 \EQ \sin^2\th_0\int_0^{2\pi}  \xi(t)\rd\eta(t)-\cos^2\th_0\int_0^{2\pi}\eta(t)\rd\xi(t)\\
 \EQ (\sin^2\th_0+\cos^2\th_0)\int_0^{2\pi} \xi(t)\rd\eta(t) =\mathcal{A}_0,
 \eeaa
where other terms in the integrant $\th_2'(t)$ vanishes due to $2\pi$ periodicity of $\vr_0(t)$ and $\vr_0'(t)$.

Going to the Poincar\'e maps corresponding to Eq. \x{rse4}
    \[ 
    \P_\la(\th_0)=\P(\th_0,\la):=\th(2\pi;\th_0,\la), \qq \th_0\in \R,
    \] 
we have obtained the expansion
    \be \lb{Pla}
    \P_\la(\th_0)=\th(2\pi;\th_0,\la) = \th_0 + \A_0 \la^2 + o(\la^2),\qq \la\to0.
    \ee
Moreover, the expansion \x{Pla} is uniform in $\th_0\in[0,2\pi]$ and in $\th_0\in \R$ as well.

Let $H_\al(\th_0):=\th_0 +2\al\pi$ denote the rigid rotation. For any $\e>0$, \x{Pla} means that there exists $\de>0$ such that
    \[
    H_{(\A_0 - \e)\la^2/2\pi}(\th_0)\equiv \th_0 + (\A_0 - \e) \la^2\le \P_\la(\th_0)\le \th_0 + (\A_0 + \e)\la^2\equiv H_{(\A_0 + \e)\la^2/2\pi}(\th_0)
    \]
for all $\th_0\in \R$ and all $|\la|<\de$. It follows from the definition of rotation number \x{rot2} that
    \[
    (\A_0 -\e)\la^2/2\pi\le \ro(\P_\la) \le (\A_0 + \e)\la^2/2\pi,\qqf |\la| <\de.
    \]
This gives the asymptotic formula \x{rho-ra} for $\ro(R)=\ro(\P_{1/R})$ as $R\to+\oo$.

The monotonicity of $\ro(R)$ follows from \x{Pla} and the assumption $\A_0\ne 0$.
\qed

Now we are going to introduce two important concepts for regular closed (escaping) curves.

 \bb{defn} \lb{CSD}
{\rm
For any $C^1$ regular closed EC $\E$, we define the {\it critical shadowing distance} 
to $\E$ as
    \[
    \ul{R}=\ul{R}(\E):= \sup \z\{R_0\in(0,+\oo): \ro(R) \equiv \om_0 =\om(\E) \mbox{ on } (0,R_0] \y\}. 
    \]
}
    \end{defn}

Let us introduce the following hypothesis $\bf{(H)}$ on $C^1$ regular closed ECs $\E$:
    \[
    \bf{(H)}: \qq \om_0=\om(\E) \ne 0,
    \andq \A_0=\A(\E)\ne 0.
    \]

    \bb{rmk} \lb{Hy}
The hypothesis $\bf{(H)}$ is always verified if $\E$ is a non trivial Jordan curve (a simple closed curve).  On the contrary,  it can be the case  that $\omega_0=0$ if $\E$ is a figure eight making by two circles touching at the origin, and $\mathcal{A}_0$ any prescribed number by varying the areas of the two disks.
    \end{rmk}

    \bb{defn}\lb{TSD}
{\rm
For any $C^1$ regular closed EC $\E$ satisfying hypothesis $\bf{(H)}$, we define the {\it turning shadowing distance}
to $\E$ as
    \[
    \overline{R}=\overline{R}(\E):=\inf \{R^*\in [\underline{R},+\infty):  \ro(R)\neq \omega_0 \,\text{and is monotone in }\, R\in (R^*,+\infty)\}.
    \]
}
    \end{defn}

Due to Theorem \ref{rho-01} and Theorem \ref{mono}, one sees that
    \be \lb{RR0}
    0<\ul{R}(\E) \le \overline{R}(\E)<+\oo
    \ee
for any $C^1$ regular closed curve $\E$ satisfying hypothesis $\bf{(H)}$. We will come back to \x{RR0} for the equality of $\ul{R} (\E) = \ol{R}(\E)$ in \S \ref{fifth}.

    \begin{thm} \lb{circ}
Let $\E$ be a regular $C^2$ closed curve with $\omega_0=\omega(\E)\ne 0$. Then an upper bound for the CSD is given by
    \be \lb{ulr-u1}
    \ul{R}(\E)\leq \frac{\ell_0}{2\pi |\omega_0|}=\frac{\mu}{|\omega_0|},
    \ee
where $``="$ holds if and only if $\E $ is a circle. Here $\ell_0= \int_0^{2\pi }\|\vr_0'(t)\|\dt$ is the algebraic perimeter of  $\E$ as before.
    \end{thm}

\Proof By reversing time if necessary, we assume that $\om_0>0$. For $R\leq \ul{R}$, one has $\ro(R)=\omega_0$. It follows from Lemma \ref{upbound} that
    \[
    \omega_0=\ro(R)\leq \frac{\ell_0}{2\pi R}.
    \]
Hence
    $$
    R\leq \frac{\ell_0}{2\pi \omega_0}=\frac{\mu}{\omega_0}, \andq \ul{R}(\E)\leq \frac{\mu}{\omega_0}.
    $$

Now we consider the case $R^*=\f{\mu}{\om_0}$. Take the normalized arc-length parametrization of the form \x{arclength} for $\E'$ and plug it into Eq. \x{rse40}. We obtain
    \be\label{rse411}\phi'=-\omega_0\cos \phi-\psi'(s),\ee
where $\phi(s)=\th(s)-\psi(s)-\frac{\pi}{2}$, and $\th=\th(s)$ solves the RSE \x{rse3}. Integration from $0$ to $2\pi$ in \x{rse411} gives
    \[
    \phi(2\pi;\phi_0)-\phi_0=-\omega_0\int_0^{2\pi}\cos \phi(s;\phi_0)ds -2\pi \omega_0\leq 0.
    \]
It follows that the rotation number $\rho_{\phi}(R^*)\leq 0$ for \x{rse411}. In particular, the Poincar\'e map $\P_{\phi}(\phi_0) :=\phi(2\pi;\phi_0)=\phi_0$ has a solution if and only if $\phi(s,\phi_0)\equiv \pm \pi$ whence $\psi'(s)\equiv \omega_0$.
Notice that the curvature of $\E$ is given by $\kappa(s)=\mu^{-1}\psi'(s)$. We conclude that $\rho_{\phi}(R^*)\leq 0$ whence $\ro(R^*)=\omega_0+\ro_{\phi}(R^*)\leq\omega_0$ with $``="$ if and only if $\E$ has constant curvature, i.e. $\E$ is a circle.\qed

\subsection{Types of shadowing curves}\lb{scs}

In the sequel, we consider regular $C^1$ closed ECs $\E$ satisfying hypothesis $\bf{(H)}$. Up to a reversal of time, we may assume $\omega_0=\omega(\E)>0$. The characterization to the types of SCs to $\E$ is based on the  observation that the range of the rotation number function $\ro(R)$ contains at least a non-trivial interval $\z(0,\, \om_0\y]$, followed by Theorem \ref{rho-01} and Theorem \ref{mono}. More precisely,  for any $\hat\ro\in (0,\omega_0]$, $\ro^{-1}(\hat\ro)\neq \emptyset$.

In order to better describe the dynamical property of the shadower, it is convenient to identify $\R^2$ to the complex plane $\C$. The EC $\vr_0(t)$ and SC $\vr(t)$ are then written as $r_0(t)\in C^1(\R, \C)$ and $r(t)\in C^1(\R, \C)$, with
    \be\label{complex}
    r(t)=r_0(t)+ Re^{i\theta(t)},
    \ee
where $\theta(t)$ satisfies the RSE \x{RSE}. Without ambiguity, the Poincar\'e map associated to \x{RSE} could be written correspondingly as
    \[
    \P: \mcc_1\to \mcc_1,\quad \P(e^{i\th}):= e^{i \P(\th)}\qqf \th\in \R.
    \]
where $\mcc_1=\{z\in \C: |z|=1\}$ is the unit circle. For any irreducible rational number $q/p$ with $p>0$, denote
    \[
    \Th_{q/p}= \Th_{q/p,\P}:= \z\{\vth\in \R: \P^p(\vth)=\vth +2 q\pi \y\}\subset \R.
    \]

The dynamical behavior of SCs follows directly from the circle dynamics of the RSE, see Theorem \ref{circledyn}.

    \bb{thm} \lb{main1}
    \begin{enumerate}
    \item Let $R \in (0,\ul{R}]$. Then $\Th_{\om_0/1} \ne \emptyset$. Moreover,

\bu for any
    \(
    \th_*\in\Th_{\om_0/1},
    \)
    \(
    r_{R,\th_*}(t)\equiv r_0(t) + R e^{i\th_{R,\th_*}(t)} 
    \)
is a $2\pi$-periodic SC to $\E$; and

\bu for any
    \(
    \th_0\in(\Th_{\om_0/1})^c,
    \)
which may be void, the SC
    \(
    r_{R,\th_0}(t)\equiv r_0(t) + R e^{i\th_{R,\th_0}(t)}
    \)
will tend to some $2\pi$-periodic SC as $t\to+\oo$ or as $t\to -\oo$.

    \item  For any irreducible rational number
    \(
    q/p\in (0, \om_0),
    \)
where $p\ge 1$, there must be some shadowing distance $\Rpq > \ul{R}$ such that
    \be \lb{SC42}
    \ro(\Rpq)= q/p.
    \ee
Accordingly, with such a shadowing distance $R=\Rpq$, any $\th_* \in \Th_{q/p}\ (\ne \emptyset)$ gives a $2p\pi$-periodic SC $r_{\rpq, \th_*}(t)$ to $\E$, and, for any $\th_0 \in (\Th_{q/p})^c$, the SC
    \(
    r_{\rpq,\th_0}(t)
    \)
will tend to some $2p\pi$-periodic SC as $t\to+\oo$ or as $t\to -\oo$.
    \end{enumerate}
    \end{thm}

\Proof These results are clear from Theorem \ref{circledyn}.  For example, in Case 1, for any $\th_*\in\Th_{\om_0/1}$, we have that
    \[
    \th(t+2\pi; \th_*) = \th(t; \P(\th_*))= \th(t;\th_* +2\om_0\pi) \equiv \th(t;\th_*) +2\om_0\pi.
    \]
Thus both $r_0(t)$ and $Re^{i \th_{R,\th_*}(t)}$ are $2\pi$-periodic. Therefore $r_{R,\th_*}(t)=r_0(t)+ Re^{i \th_{R,\th_*}(t)}$ is necessarily $2\pi$-periodic.

For Case 2, the existence of $\rpq$ in \x{SC42} is an immediate result of the properties of rotation numbers. Moreover, in this case, for any $\th_*\in \Th_{q/p}$, $\rpq e^{i \th_{\rpq,\th_*}(t)}$ is now $2p\pi$-periodic. The others are similar.\qed

    \bb{thm} \lb{main2}
For any irrational number
    \(
    \ro_0 \in (0, \om_0)\bs \Q,
    \)
there must be some shadowing distance $\Rrh > \ol{R}$ such that
    \be \lb{SC52}
    \ro(\Rrh)=\ro_0.
    \ee
Moreover, with such a shadowing distance $R_{\ro_0}$, any SC $
   \mcs_{\Rrh,\th_0}$
 is dense in the shadowing domain 
\be \lb{Dr}
    \D_{\Rrh}:=\z\{r_0(t) + \Rrh z: t\in\R, \, z\in \mcc_1 \y\}
    \ee
that is
   \be\lb{SC54}
    \ol{\z\{ r_0(t) + \Rrh e^{i \th_{\Rrh,\th_0}(t)}:\ t\in \R\y\}}=\D_{\Rrh},
    \ee

    \end{thm}

\Proof As before, the existence of $\Rrh$ in \x{SC52} follows from the continuity property of $\ro(R)$.

Denote $\P_0:= \P_{\Rrh}$. One has
    \be\lb{den}\bb{split}
    \th_{\Rrh,\th_0}(t+2 k\pi) &\equiv \P_0^k (\th_{\Rrh,\th_0}(t))\q \mbox{on }\R, \\
    e^{i\th_{\Rrh,\th_0}(t+2 k\pi)} &\equiv \P_0^k (e^{i \th_{\Rrh,\th_0}(t)})\q \mbox{on }\mcc_1.
    \end{split}
    \ee
Since $\P_0:\mcc_1\to\mcc_1$ is an analytic diffeomorphism of the circle and $\ro(\P_0) =\ro_0$ is irrational, it is known from the Denjoy theorem that $\P_0$ is minimal in $\mcc_1$
    \be \lb{den1}
    \ol{\z\{\P_0^k(w): k\in \Z\y\}}=\mcc_1\qqf w\in \mcc_1.
    \ee

Let $\th_0\in \R$ and $t\in \R$ be arbitrarily given. To apply \x{den1}, we choose $w=w_t:= e^{i \th_{\Rrh,\th_0}(t)}\in \mcc_1$. Then, for any point $z\in \mcc_1$, there must be a sequence $\{k_m=k_m(w_t,z)\}_{m\in \N}\subset \Z$ such that
    \[
    \lim_{m\to+\oo} \P_0^{k_m} (e^{i w_t})=z.
    \]
By \x{den}, we know that, as $m\to+\oo$,
    \[
    r_0(t+ 2k_m\pi) + \Rrh e^{i \th_{\Rrh,\th_0}(t+2k_m\pi)}\equiv r_0(t) + \Rrh \P_0^{k_m}(w_t) \to  r_0(t) + \Rrh z.
    \]
That is, any point $r_0(t) + \Rrh z\in \D_{R_{\ro_0}}$ can be approximated by points from $\mcs_{\Rrh,\th_0}$. Due to the trivial inclusion $\mcs_{\Rrh,\th_0}\subset \D_{R_{\ro_0}}$, we have proved the density \x{SC54}.

Since the set $\D_{\Rrh}$ in \x{Dr} is parameterized by $t\in \R$ and $z\in \mcc_1$, it is easy to see that $\D_{\Rrh}\subset \C$ is a planar domain.
\qed

In the above proof, the Denjoy theorem is crucial.

\subsection{Regularity of Shadowing curves}\lb{scs2}
Given a regular smooth closed escaping curve $\E$, we will show that when the shadowing distance $R$ is small, the shadowing curves remains regular; while for large $R$, turning points (cusps) appear.

\begin{defn} Let $\vr: \R\to \R^2$ be a $C^2$ curve.  A point $\vr(t_0)$ is called a singular point if $\vr'(t_0)=0$; it is called a turning point or an (ordinary)\footnote{The term ordinary cusp, means a singularity of type $3/2$, i.e. locally of the normal form $(t^2,t^3)$.} cusp, if $\vr'(t_0)=0$ and $\vr''(t_0)\neq 0$.
\end{defn}

\begin{defn}(\cite{DC57},\cite{FN90}) Let $\vr: \R\to \R^2$  be a $C^2$ curve. The parallel curve or offset curve $\vp_d: \R\to \R^2$ of $\vr_0$ at a signed distance $d$ is defined as 
\[\vp_d(t)=\vr_0(t)+dN(t), \quad \text{with unit outer normal}\, N(t)\]
The curve $\vp_d$ is called outer (resp. inner) parallel curve if $d>0$ (resp. $d<0$).
\end{defn}

Indeed, the parallel curves $\vp_{\pm d}$ are the envelops of a family of congruent circles of radius $d$ centered on the progenitor curve $\vr_0$.  Physically, the parallel curve is Huygen's wave front  for a source, of the form given by the progenitor curve, emitting waves with unit speed.

\begin{lem}\label{sig} Let $\vr_0: \R\to \R^2$  be a regular $C^2$ closed escaping curve, then singular points (if exist) of any shadowing curve with shadowing distance $R>0$ lie on the parallel curves $\vp_{\pm R}$ of $\vr_0$.
\end{lem}
\Proof Let  $\vr_0: \R\to \R^2$ be an escaping curve.  It follows from \x{eq1} and \x{kt} that the shadowing curve $\vr(t)$ satisfies
\be\label{sing2}\vr'(t)=\alpha(t)(\vr(t)-\vr_0(t)),\quad \text{with}\,\alpha(t)=\frac{\vr_0'(t)\cdot (\vr(t)-\vr_0(t))}{\|\vr(t)-\vr_0(t)\|^2}\ee
Hence $\vr'(t)=0$ iff $\alpha(t)=0$, i.e. $\vr(t)-\vr_0(t))=d(t)N(t)$ for some constant $d(t)$, where $N(t)$ is the unit outer normal of $\vr_0(t)$. Since $\|\vr(t)-\vr_0(t)\|\equiv R$, it follows $d(t)=\pm R$.
\qed

\medskip
In the following, in order to have further description of the existence of turning points of a shadowing curve, we restrict ourself to escaping curves which are strictly convex simple closed curves. Notice that the rotation index $\omega$ of a simple closed curve is always $\pm 1$, up to changing orientation, we may always assume that the $\omega(\E)=1$ when given a simple closed escaping curve $\E$.


\begin{lem} \label{para}Let $\vr_0:\R\to\R^2$ be a $C^3$ strictly convex simple closed curve, $R_{min}$ and $R_{max}$ are the minimum and maximum of the radius of curvature of $\vr_0$ respectively. Then for $|d|<R_{min}$ or $|d|>R_{max}$, the parallel curve $\vp_d$ of $\vr_0$ with signed distance $d$ is a $C^2$ convex simple closed curve.
\end{lem}

\Proof Without loss of generality, we may put $\vr_0$ in the arc-length parametrization $\vr_0(s)$. Let  $T(s)=\vr_0'(s)$ and $N(s)$ denote the unit tangent vector and outer normal vector respectively, then
\[\vp_d'(s)=\vr_0'(s)+d N'(s)=(1-d\kappa_0(s))T(s)=:B_d(s)T(s),\]
where $\kappa_0(s)$ is the (signed )curvature of $\vr_0$. Write $T(s)=(\cos \varphi(s),\sin\varphi(s))$,   then $\kappa_0(s)=\varphi'(s)$ is always positive or negative by convexity of $\vr_0$. For $|d|<R_{min}$ or $|d|>R_{max}$, the (signed) curvature of $\vp_d$, followed by \x{curvature}, 
\[\kappa_d(s)=\frac{\varphi'(s)}{B_d(s)}\]
is always positive or negative, hence $\vp_d(s)$ is convex, having no singularity.  In particular, since 
\[\vp_d'(s)=B_d(s)((\cos \varphi(s),\sin\varphi(s)),\]
 the turning angle $\varphi(s)$ is monotone and $\vp_d$ has the same rotation index, i.e. $\pm 1$, as $\vr_0$, it follows that the curve $\vp_d$ is simple.
\qed
 
\begin{thm} Let $\E$ be a $C^3$ strictly convex simple closed escaping curve,  $R_{min}$ and $R_{max}$ are the minimum and maximum of the radius of curvature of $\E$ respectively .Given a shadowing distance $R>0$, the following hold:
\begin{enumerate}
\item If $R\in (0, R_{min})$,  then
 there exist at least one regular $C^2$  closed shadowing curve inside $\E$; moreover,
any shadowing curve has at most one turning point;
\item If $R>R_{max}$ and $\ro_{\E}(R)=1-q/p$, with $p,q$ relatively prime, then any closed shadowing curve with minimal period $2p\pi$ have exactly $2q$ turning points; 
\item If $R>R_{max}$ and $\ro_{\E}(R)$ is irrational, then any shadowing curve has infinitely many turning points.

\end{enumerate}

\end{thm}

\Proof Let $\E$ be parametrized by $\vr_0(t)$ as in \x{pa1} and \x{reg}.   
A point in the shadowing curve $\vr(t):=\vr_R(t;\th_0)$ is a singular point iff $\alpha(t)=0$ where $\alpha(t)$ is given in \x{sing2}. By \x{mpc} and \x{ec23}, 
\[\alpha(t):=\alpha_{R,\th_0}(t)=\frac{B(t)}{R}\cos(\th(t)-\psi(t)),\]
where $\th(t):=\th_R(t;\th_0)$ is the solution of \x{RSE} with initial condition $\th(0)=\th_0$. Let $\phi(t):=\th(t)-\psi(t)-\pi/2$, then $\phi(t):=\phi_R(t;\phi_0)$ is a solution of \x{rse40} with $\phi_0=\th_0-\psi(0)-\pi/2$.  It follows that $\alpha(t)=0$ iff $\phi(t)=k\pi$ for some $k\in \Z$.

1, For $0<R<R_{min}=\min_{t\in [0,2\pi]}\frac{B(t)}{|\psi'(t)|}$,
 the vector field in \x{rse40}
\[F_R(t,\phi):=-\frac{1}{R}B(t)\cos\phi-\psi'(t)\] satisfies
\[F_R(t,0)<0,\quad F_R(t,-\pi)>0,\quad t\in \R.\]
and  the $2\pi$-periodic solutions $\phi(t;\phi_R^{\pm})$ defined in \x{persol1} never obtain $0$ and $\pm \pi$, hence the corresponding closed shadowing curve $\vr_{R,\th_R^{\pm}}(t)$ has no singular point.
 
 Moreover, from the geometric explanation of the vector field $F_R(t,\phi)$, one sees that any integral curve $\{(t,\phi_R(t;\phi_0)))\}_{t\in\R}$ of \x{rse40} can intersect the lines $\{(t,\phi=k\pi)|t\in\R, k\in\Z\}$ at most once. Hence for any $\th_0$, there exists at most one $t_0$ such that touch that $\phi(t_0)=k\pi$ for some $k\in\Z$, whence $\alpha_{R,\th_0}(t_0)=0$.  

2,  If $\ro_{\E}(R)=1-q/p$, then $\ro_{\phi}(R)=-q/p$ for \x{rse40}, since the rotation index of $\E$ is assumed to be $1$. Any closed shadowing curve is given by some $\vr_R(t;\th_*)$ such that the corresponding solution $\phi_R(t;\phi_*)$ of \x{rse40} satisfies
\be\label{pq}\phi_R(t+2p\pi ;\phi_*)=\phi_R(t;\phi_*)-2q\pi,\quad \phi_*=\th_*-\psi(0)+\pi/2\ee

Let $R>R_{max}$, then $F_R(t,\phi)<0$ for all $t,\phi$. Hence $\phi_R(t,\phi_*)$ is strictly decreasing. Suppose $\phi_R(t,\phi_*)$ obtains a multiple of $\pi$ at $t=t_{k_0}$,  i.e. $\phi_R(t_{k_0},\phi_*) =k_0\pi$ for some integer $k_0$. From \x{pq}, there exists exactly $2q$ points in the interval $[t_{k_0},t_{k_0}+2p\pi]$, denoted by $0\leq t_{k_0}<t_{k_0+1}<\cdots<t_{k_0+2q}=t_{k_0}+2p\pi$, such that $\phi_R(t_{k_0+j};\phi_*)=(k_0-j)\pi$, $j=0,\cdots,2q$.
These $\vr(t_k):=\vr_R(t_k;\th_*)$ , $k=k_0,\cdots,k_0+2q$ are singular points, we claim that they are different. 

Indeed, by Lemma \ref{sig}, $\vr_R(t_k)$ lie on the parallel curves $\vp_{\pm R}$ of $\vr_0$, with $+R$ for even $k$  and $-R$ for odd $k$.  By Lemma \ref{para}, the curves $\vp_{\pm R}$ has no self-intersection, and $\vp_R$ does not intersect $\vp_{-R}$.  Therefore, if there are two points $\vr_R(t_k)=\vr_R(t_{k'})$, then $t_{k'}=t_{k}+2m\pi$ and $k'-k=2m'\pi$ for some positive integers $m\leq p, m'\leq q$.  Let $\tilde{\phi}(t):=\phi_R(t+t_k,\phi_*)$, then $\tilde{\phi}(t)$ is a solution of the $2\pi$-periodic equation \x{rse40} with vector field $\tilde{F}(t,\phi):=F_R(t+t_k,\phi)$, and the rotation number of $\tilde{\phi}(t)$ is $-q/p$. Denote $\tilde{\phi}_0=\tilde{\phi}(0)$, we have
\[\tilde{\phi}(t+2m\pi;\tilde{\phi}_0)=\tilde{\phi}(t;\tilde{\phi}(2m\pi;\tilde{\phi}_0))=\tilde{\phi}(t;\phi_R(t_{k'};\phi_*))=\tilde{\phi}(t;\tilde{\phi}_0-2m'\pi)=\tilde{\phi}(t;\tilde{\phi}_0)-2m'\pi\]
Hence the rotation number of $\tilde{\phi}$ is $-m'/m$, therefore $m=p$, $m'=q$. We conclude that $\vr(t_k),k=k_0,\cdots, k_0+2p$ are exactly $2q$ different points.

3. If $\ro_{\E}(R)$ is irrational, then $\ro_{\phi}(R)$ is irrational, in particular, nonzero. It follows that the range of the solution $\phi_R(t;\phi_0))$ of \x{rse40} is $\R$, thus there exists $t_k$ such that $\phi(t_k)=k\pi$  for any $k\in \Z$. These $\vr(t_k)=\vr_R(t_k;\th_0)$ are different singular points for different $k$, otherwise by the same arguments in as above, $\vr(t)$ is periodic, which contradict with the irrationality of $\ro_{\phi}(R)$.

Finally, if $R>R_{max}$ or $R<R_{min}$, then at a singular point $\vr(t_k)$ of a shadowing curve, such that $\phi(t_k)=k\pi$, we have
\[\|\vr''(t_k)\|=R|\alpha'(t_k)|=B(t_k)|\phi'(t_k)|=B(t_k)|F_R(t_k,\phi(t_k))|\neq 0,\]
hence $\vr(t_k)$ is a turning point.
\qed

\section{Shadowing Curves to the Unit Circle and Ellipses}
\setcounter{equation}{0} \lb{fourth}
This section is devoted to more analytic and numerical analysis on the rotation number and theorems concerning dynamical behaviors and regularities of shadowing curves in section \ref{scs}, \ref{scs2}, when the escaping curve is a circle or an ellipse.
\subsection{SCs to the unit circle} \lb{sc-c1}

In this subsection, we give a detailed analysis for shadowing curves to the unit circle
    \[
    \mcc_1: \q \vr_0(t)=(\xi(t),\eta(t)) =(\cos t, \sin t).
    \]
That is, the escaper starts at the point $(1,0)$ and then runs along the unit circle $\mcc_1$, centered at the origin, anti-clockwise at a constant speed $1$. With the circular EC $\mcc_1$, the SE and RSE are respectively
    \bea\lb{rse1}
    \vr' \EQ \f{(-\sin t,\cos t)\d \vr}{\|\vr-(\cos t,\sin t)\|^2}(\vr-(\cos t,\sin t)),\\
    \lb{rse2}
    \th' \EQ -\f{1}{R}(\sin t\sin \th +\cos t\cos \th)= - \f{\cos(\th-t)}{R}.
    \eea
Here SE \x{rse1} is a planar system of ODEs, while RSE \x{rse2} is an ODE on the line or on the circle.
For RSE \x{rse2}, by letting
    \be \lb{ths}
    \phi:= \th-t,
    \ee
problem \x{rse2}-\x{th0} is transformed into the following initial value problem for $\phi=\phi(t)$
    \bea \lb{phi1}
    \phi' =\f{\rd\phi}{\dt} \EQ  - \f{\cos \phi+R}{R},\\
    \lb{phi0}
    \phi(0) \EQ \th_0.
    \eea
All solutions $\phi(t)=\phi_\rth(t)$ of \x{phi1}-\x{phi0} can be calculated explicitly.

In the sequel we will adopt the complex form \x{complex} for EC and SCs.
Now the EC is $r_0(t) = e^{i t}$ and SCs are given by
    \be \lb{r1t}
    r(t) = e^{i t}+ Re^{i(t+\phi(t))} \equiv e^{i t} (1+R e^{i \phi(t)}).
    \ee
For later uses, we can use Eq. \x{phi1} to deduce from \x{r1t} that
    \be \lb{rt'}
    r'(t) \equiv e^{i(t+\phi(t))} \sin \phi(t),\andq
    r''(t) \equiv -\f{1}{R} e^{i(t+\phi(t))}(e^{i\phi(t)}+R)\cos \phi(t).
    \ee
The circle of radius $b>0$ centered at the origin is denoted by
    \[
    \bb{split}
    \mcc_b & := \{\vr\in \R^2: \|\vr\|= b\}=\{ z\in \C: |z|=b\}.
    \end{split}
    \]

 First note that for $\E=\mcc_1$, its radius of curvature is constant $1$, hence $R_{min}=R_{max}=1$; the rotation index is $\omega_0=1$.
\subsubsection{SCs with $R\in(0,1)$}\lb{sc-c2}

We first consider shadowing distances $R\in(0,1)$. In this case, RSE \x{phi1} has two geometrically different equilibria. In order to be consistent with the asymptotic stability, they are labelled as $\phi=\th_R^+$ and $\phi=\th_R^-$, where
    \be \lb{ph0}
    \th_R^+ = -\th_R^-:=-\arccos(-R)\equiv -\z(\pi/2+\arcsin R\y).
    \ee
Hence $\th_R^+\in(-\pi,-\pi/2)$ and $\th_R^-\in(\pi/2,\pi)$. These equilibria $\th_R^\pm$ yield the periodic SCs
    \bea \lb{Rr}
    r_{R,\th_R^\pm}(t) \EQ e^{i t} (1+R e^{i \th_R^\pm})
    \equiv \sqrt{1-R^2}\, e^{i(t\mp \arcsin R)}.
    \eea
The initial points of periodic SCs \x{Rr} are
    \[
    r_{R, \th_R^\pm}(0) = 1-R^2 \mp i R \sqrt{1-R^2}\in \z\{r\in \C: \z|r-{1}/{2}\y|={1}/{2}\y\}=: \hat\mcc.
    \]
In fact, $r_{R, \th_R^+}(0)$ and $r_{R, \th_R^-}(0)$ are located on the lower and the upper semi-circles of $\hat\mcc$ respectively. SCs $r_{R,\th_R^\pm}(t)$ have the minimal period $2\pi$. Their trajectories are the same circle  $\mcc_{R_1}$, $R_1:=\sqrt{1-R^2}$, but with different phases.

From the above reasoning, SCs \x{Rr} are called the circular SCs to $\mcc_1$. In fact, these circular SCs can be constructed using  elementary geometry. See Figure \ref{sccircles}.

\begin{figure}[ht]
\centering
\includegraphics[width=6.5cm, height=6cm]{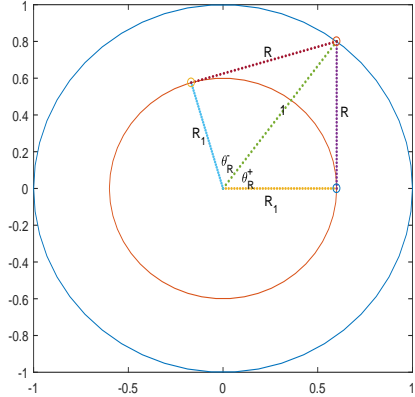}
\caption{Circular SCs $\mcs_{R,\th_R^\pm}$ to the unit circle $\mcc_1$, where $R=4/5$ and $R_1=\sqrt{1-R^2}=3/5$. }
\label{sccircles} 
\end{figure}

Due to the labelling as in \x{ph0}, the equilibria $\th_R^+$ and $\th_R^-$ of RSE \x{phi1} are respectively positively  and negatively stable. Going to SCs $r_\rth(t)$, we have from \x{phi1}---\x{r1t} the following results. 

    \bb{thm} \lb{Rxy1}
Let $R\in(0,1)$ and consider SCs $r_\rth(t)$.

\bu If $\th_0=\th_R^\pm$, then $r_{R,\th_R^\pm}(t)$ are the circular SC $\mcc_{\sqrt{1-R^2}}$ with different phases.

\bu If 
$\th_0\ne \th_R^\pm$, then $\mcs_\rth$ has the unique turning point at
    \be \lb{taur}
    t=\tau(\th_0):= \z\{ \ba{ll}
    \f{2R}{\sqrt{1-R^2}}\, \arctanh \z( \sqrt{\f{1-R}{1+R}} \tan \f{\th_0}{2}\y) & \mbox{ for } \th_0\in(\th_R^+,\th_R^-),\\
    \f{2R}{\sqrt{1-R^2}}\, \arctanh \z( \sqrt{\f{1+R}{1-R}}\cot\f{\th_0}{2}\y) & \mbox{ for } \th_0\in(\th_R^-,\th_R^++2\pi).
    \ea\y.
    \ee
Moreover, as $t\to \pm \oo$, the SC $r_\rth(t)$ is asymptotic to the circular SCs $r_{R,\th_R^\pm}(t)$ respectively.
    \end{thm}

\Proof For the case $\th_0=\th_R^\pm$, the results have been stated in \x{Rr}.

To obtain \x{taur}, we need only to choose $\th_0$ from $(\th_R^+, \th_R^-) \cup  (\th_R^-,\th_R^+ +2\pi)$. By \x{rt'}, the irregular time $\tau$ of $r_\rth(t)$ satisfies $\sin\phi_\rth(\tau)=0$. Precisely, we have the following two cases.

\bu For case $\th_0\in(\th_R^+,\th_R^-)$, one has $\phi_\rth(\tau)=0$, and by \x{phi1}-\x{phi0},
    \[
    \tau = -\int_{\th_0}^0 \f{R \rd\phi}{R+\cos \phi} = \f{2R}{\sqrt{1-R^2}} \arctanh \z( \sqrt{\f{1-R}{1+R}} \tan \f{\th_0}{2}\y).
    \]

\bu For case $\th_0\in(\th_R^-,\th_R^++2\pi)$, one has $\phi_\rth(\tau)=\pi$, and by \x{phi1}-\x{phi0},
    \beaa
    \tau \EQ -\int_{\th_0}^\pi \f{R \rd\phi}{R+\cos \phi}
    =\f{2R}{\sqrt{1-R^2}} \arctanh \z( \sqrt{\f{1+R}{1-R}}\cot\f{\th_0}{2}\y).
    \eeaa
These are what stated in \x{taur}.

At $t=\tau$, one has from \x{rt'} that $r'(\tau)=0$  and
    \[
    \z|r''(\tau)\y| =\z| -\f{1}{R} e^{i(\tau+\phi_\rth(\tau))}(e^{i\phi_\rth(\tau)}+R)\cos \phi_\rth(\tau)\y|\ge \f{1-R} R>0.
    \]
Hence $r_\rth(t)$ must have a turning point at $\tau(\th_0)$. See Figure \ref{sc-circle5}.

Furthermore, let $\th_0\ne \th_R^\pm$. Then, as $t\to \pm \oo$,
    \[
    \z| r_\rth(t)-r_{R,\th_R^\pm}(t)\y| \equiv R \z|e^{i \phi_\rth(t)}- e^{i \phi_{R,\th_R^\pm}(t)}\y|\to 0.
    \]
That is, $r_\rth(t)$ is attracted to $r_{R,\th_R^\pm}(t)$ as $t\to \pm \oo$ respectively. \qed

For these SCs, see Figure \ref{sc-circle5}.

\begin{figure}[ht]
\centering
\includegraphics[width=7.5cm, height=4.8cm]{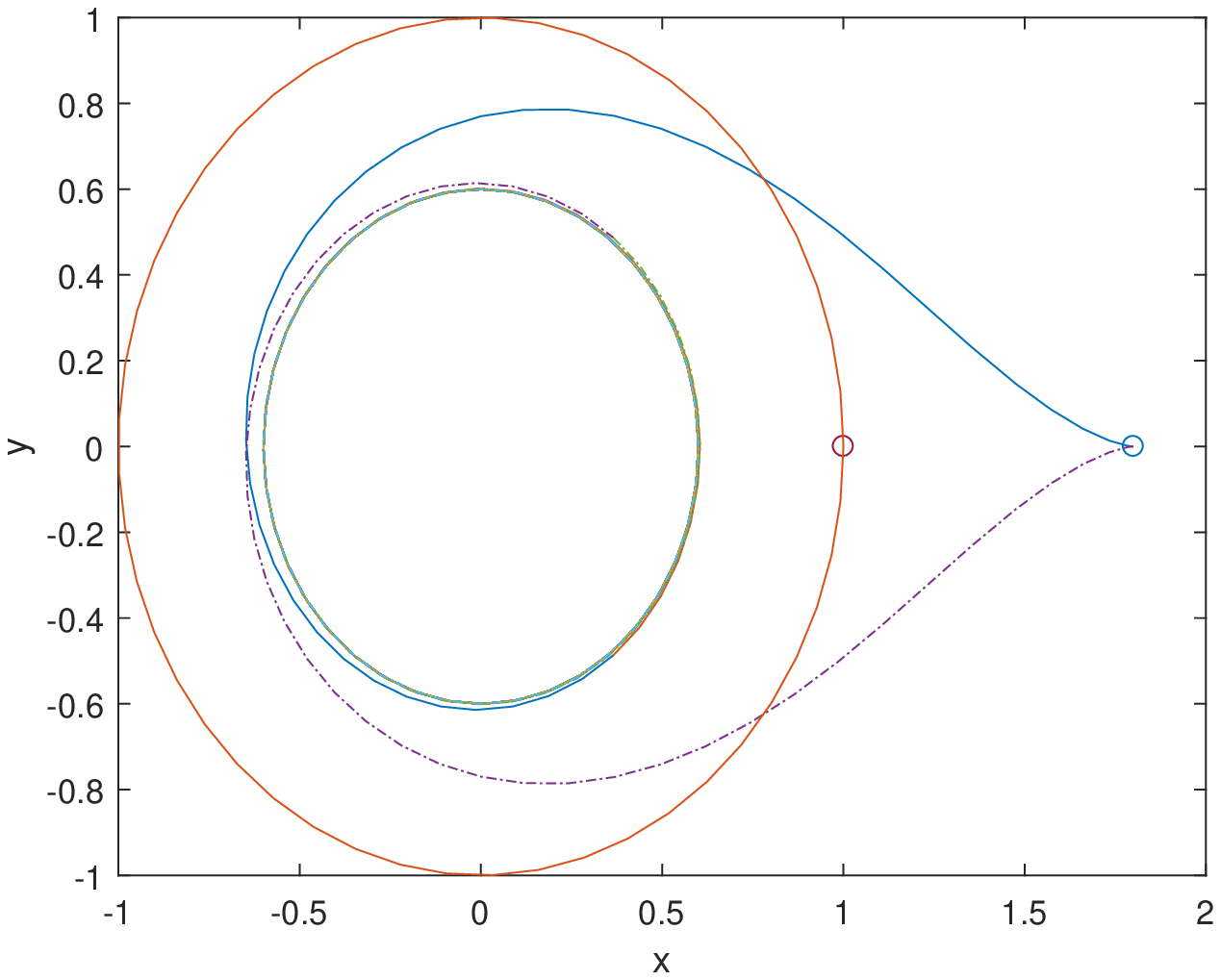}\hspace{1cm}
\includegraphics[width=6cm, height=5.8cm]{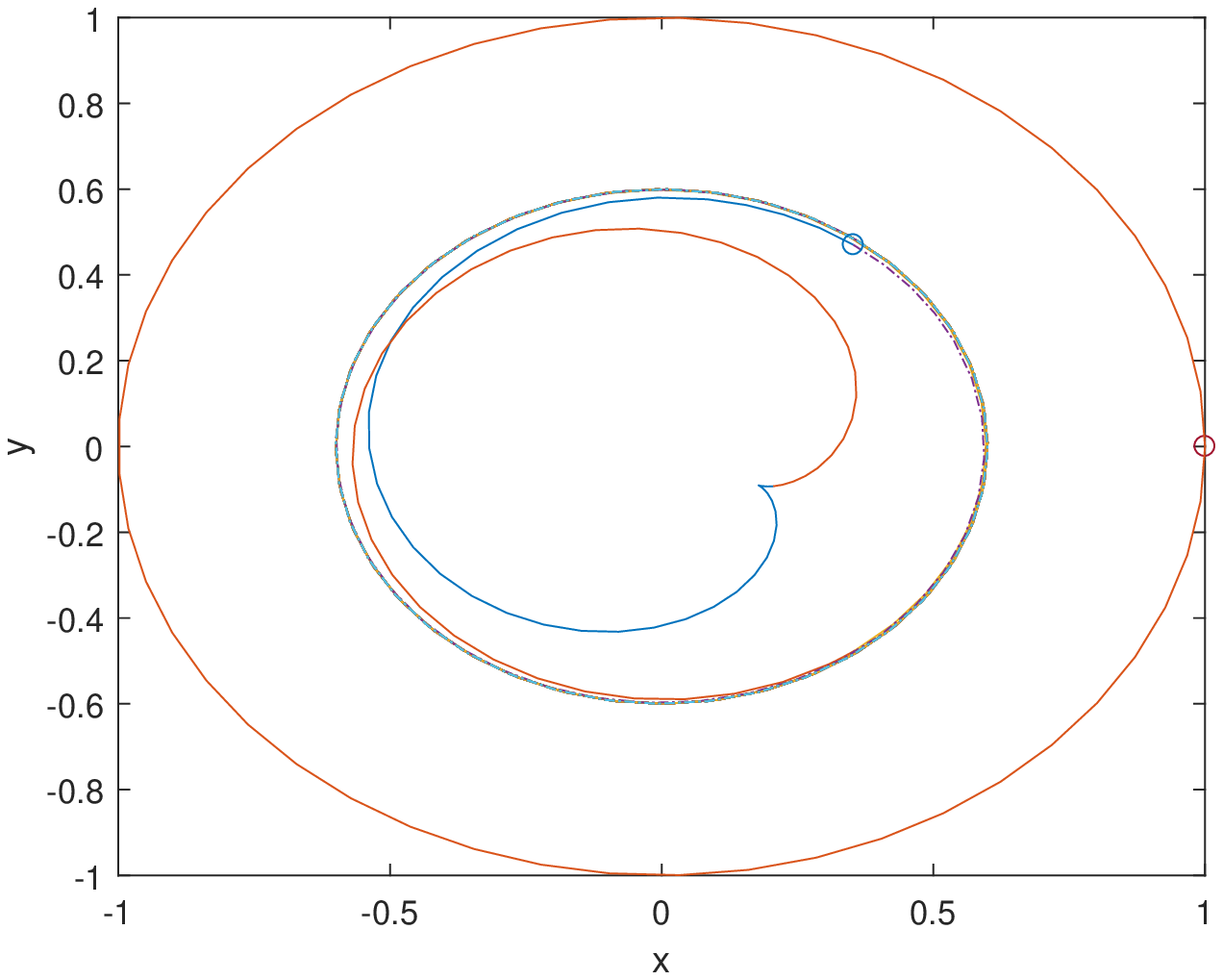}
\caption{SCs $\mcs_{4/5,0}$ (left) and $\mcs_{4/5,4\pi/5}$ (right). Here $\circ$ denotes the initial points,
the solid curve --- is for positive time, and the dashed curve -$\d$-$\d$  is for negative time.}
\label{sc-circle5} 
\end{figure}

\subsubsection{SCs with $R=1$}\lb{sc-c3}

Next we consider the case $R=1$. That is, the shadowing distance $R$ is precisely equal to the radius of $\mcc_1$. In this case, RSE \x{phi1} is
    \be \lb{phir1}
    \f{\rd\phi}{\dt} = - 2 \cos^2 \f{\phi}{2}.
    \ee
Eq. \x{phir1} has only one geometrically different equilibrium, say $\phi=\pi$. Going to \x{r1t}, this is $r_{1,\pi}(t) \equiv 0$, an equilibrium of SE \x{rse1} or a constant SC. It means that the shadower is just always standing at the center of the escaping circle. The other SCs are as follows.

    \bb{thm} \lb{Rdy1}
Let $R=1$ and $\th_0\in (-\pi,\pi)$. Then SC $r_{1,\th_0}(t)$ of \x{rrtt} is given by
    \be \lb{rrtt}
    \mcs_{1,\th_0}: \q r_{1,\th_0}(t)
    \equiv \f{e^{i t} (1+e^{i \th_0} - i t(\cos \th_0+1) )}{1-t \sin \th_0 +t^2 \cos^2(\th_0/2)}.
    \ee
Moreover, $r_{1,\th_0}(t)$ has a turning point at $t=\tau(\th_0):= \tan(\th_0/2)$, and, as $t\to \pm \oo$, $r_{1,\th_0}(t)$ is asymptotic to the equilibrium  $r_{1,\pi}(t)\equiv 0$.
    \end{thm}

\Proof Let $\th_0\in (-\pi,\pi)$. By \x{phir1}-\x{phi0}, $\phi(t)=\phi_{1,\th_0}(t) \in (-\pi,\pi)$ is determined by
    \be \lb{jf31}
    \tan \f{\phi}{2} =\tan\f{\th_0}{2}-t.
    \ee
Going to \x{r1t}, one can use \x{jf31} to simplify $r_{1,\th_0}(t)= e^{i t} \z( 1+e^{i \phi(t)}\y)$ to \x{rrtt}.

It follows from \x{rt'} that $r(t)$ is singular at time $t$ if and only if
    \(
    \sin \phi(t) =0.
    \)
Since $\phi(t)$ is within $(-\pi,\pi)$, one has $\phi(t)= 0$, and then by \x{jf31}, the unique time $t$ is $\tau(\th_0)= \tan(\th_0/2)$. Moreover, by using \x{rt'}, one has also
    \(
    r''(\tau(\th_0)) = -2 e^{i \tau(\th_0)}=- 2e^{i \tan(\th_0/2)}\ne 0.
    \)
Thus $r_{1,\th_0}(t)$ is actually a turning point at $t=\tau(\th_0)$. See Figure \ref{sc-circle0}.

As $t\to\pm\oo$, it follows from formula \x{rrtt} or directly from Eq. \x{phir1} that $\phi_{1,\th_0}(t) \to \mp\pi$. Therefore
    \(
    r_{1,\th_0}(t) = e^{i t} (1+ e^{i \phi_{1,\th_0}(t)})\to e^{i t} (1+ e^{\mp i \pi})\equiv 0.
    \)
\qed

\begin{figure}[ht]
\centering
\includegraphics[width=6.5cm, height=4.2cm]{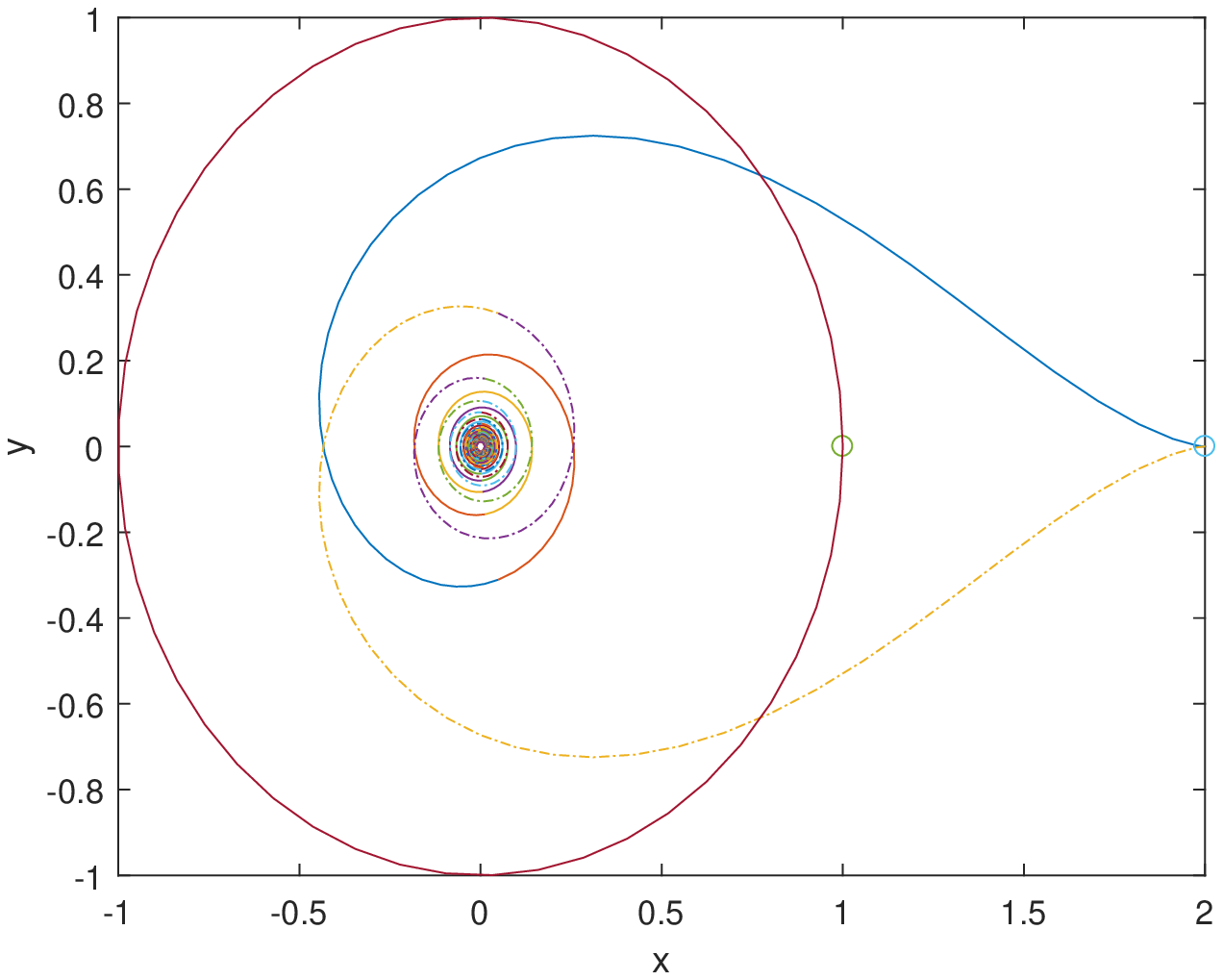}\hspace{1cm}
\includegraphics[width=6.5cm, height=7cm]{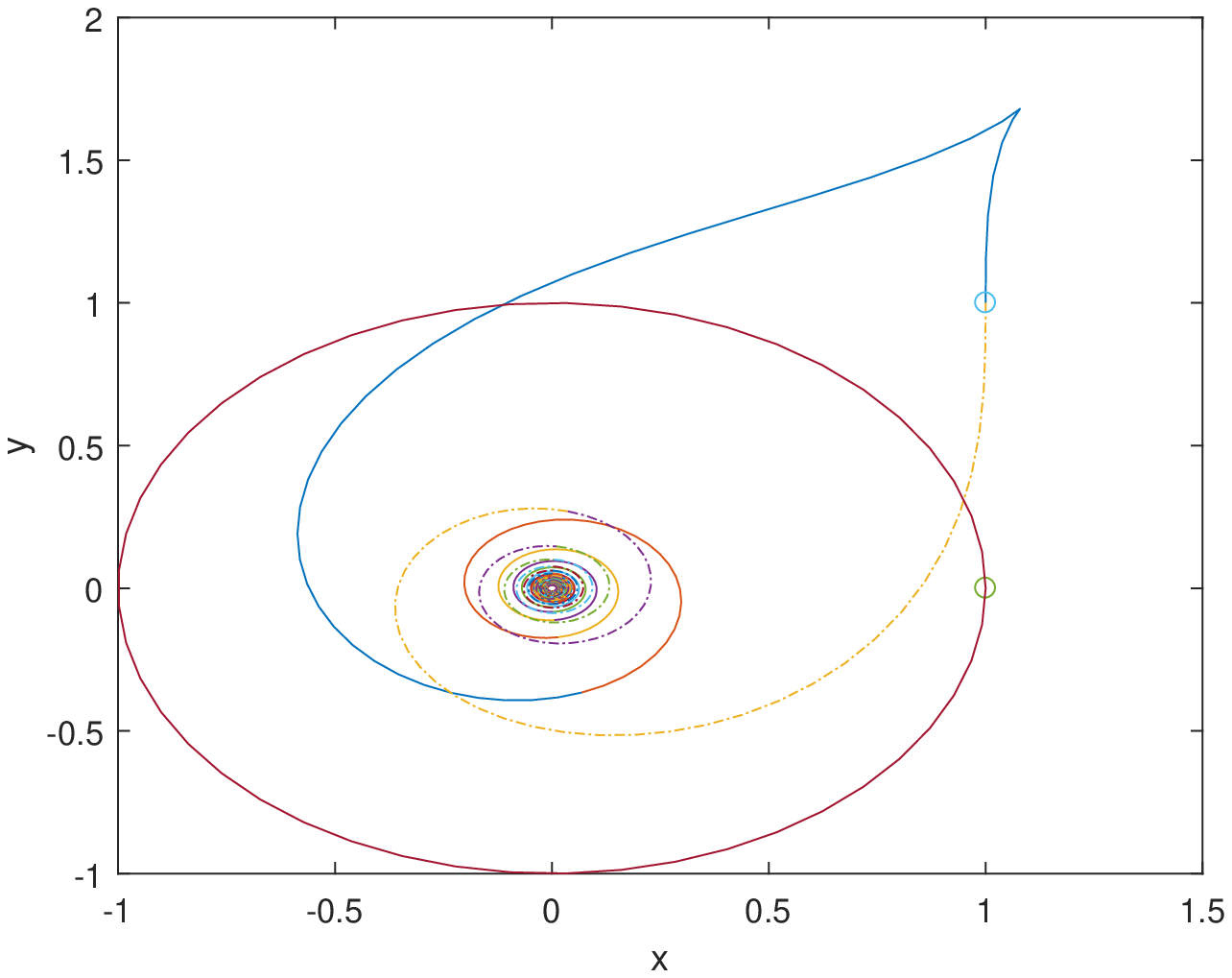}
\caption{SCs $\mcs_{1,0}$ (left) and $\mcs_{1,\pi/2}$ (right). Here $\circ$ denotes the initial points,
the solid curve --- is for positive time, and the dashed curve -$\d$-$\d$  is for negative time.}
\label{sc-circle0} 
\end{figure}

The explanation to the SCs of Theorem \ref{Rdy1} is as follows. Geometrically, each $\mcs_{1,\th_0}$ has exactly one turning point at time $\tau(\th_0)$ and looks  like two spirals defined mainly using positive and negative times $t$ respectively. Moreover, the shadower will finally approach to the center of the escaping circle. See Figure \ref{sc-circle0}.

\subsubsection{SCs with $R\in(1,\oo)$} \lb{sc-c4}

At last we consider shadowing distances $R\in(1,\oo)$. In this case, RSE \x{phi1} has no equilibrium. To solve Eq. \x{phi1}, let us define
    \be \lb{sec321}
F(\phi)=F_R(\phi):=-\int_0^\phi \f{R\rd\phi}{\cos \phi+R}, \qq \phi \in \R.
    \ee
It is a strictly decreasing, smooth, odd function on $\R$. Using the function $F(\phi)$ of \x{sec321}, the solution $\phi(t)= \phi_\rth(t)$ of problem \x{phi1}-\x{phi0} satisfies
    $
    F(\phi)-F(\th_0)\equiv t,
    $
i.e.
    \be \lb{sec322}
    \phi_\rth(t) \equiv F^{-1}_R\z( F_R(\th_0)+t\y),\qq t\in \R.
    \ee
In order to be consistent with the rotation number of Eq. \x{phi1}, let us introduce
    \be \lb{vr}
    \rho = \rho_R:= -\f{\sqrt{R^2-1}}{R}\in (-1,0)\qq \mbox{for } R\in(1,\oo).
    \ee
Then $F_R(\phi)$ is explicitly given by
    \be \lb{sec34}\bb{split}
    F(\phi) & \equiv \f{2}{\rho} \arctan\z( \sqrt{\f{R-1}{R+1}} \tan \f{\phi}{2}\y)\qqf \phi\in(-\pi,\pi),\\
    F(\pm \pi) & = \pm \f{\pi}{\rho} ,\\
    F(\phi+2\pi) & \equiv F(\phi) +  \f{2\pi}{\rho} \qqf \phi\in\R.
    \end{split}
    \ee
From \x{sec322}---\x{sec34} we deduce that $\phi(t)$ satisfies
    \be \lb{ph-p}
    \phi(t+k \pi/\rho) \equiv \phi(t)+k\pi, \andq \phi(t+ 2k\pi/\rho) \equiv \phi(t)+2k\pi
    \ee
for all $t\in \R$ and all $k\in\Z$.

Now we calculate the rotation number of the RSE \x{rse2}, including the case $R\in(0,1]$.

\bb{lem} \lb{rhoc}
The rotation number of the RSE \x{rse2} is 
    \be \lb{rhor3}
    \ro(R)=
    \z\{ \ba{ll}
    1 & \mbox{ for } R\in(0,1],\\
    1-\f{\sqrt{R^2-1}}{R}  
    & \mbox{ for } R\in[1,\oo).
    \ea\y.
    \ee
\end{lem}

\Proof We use $\ro_\phi(R)$ to denote the rotation number of Eq. \x{phi1}. Due to the change \x{ths} of variables, the rotation number $\ro(R)$ of the RSE \x{rse2} is
    \be \lb{rh1}
    \ro(R)\equiv 1+ \ro_\phi(R).
    \ee

For $R\in(0,1]$, Eq. \x{phi1} has constant solutions $\phi=\phi_{R,\th_R^\pm}(t) \equiv \th_R^\pm$. Hence its rotation number is
    \be \lb{rh2}
    \ro_{\phi}(R)= \lim_{t\to \oo} \f{\phi_{R,\th_R^\pm}(t)  -\th_R^\pm}{t}=0.
    \ee

Next let $R\in(1,\oo)$. By using  the number $\rho=\rho_R\in (-1,0)$ in \x{vr}, we can apply equalities in \x{ph-p} with any initial angle $\th_0$. By choosing times
    \[
    \tau_k:= k \pi/(-\rho)\to +\oo \qq\mbox{as } k\to+\oo,
    \]
one has from \x{ph-p} that
    \[
    \phi(\tau_k) \equiv \phi(0)-k\pi =\th_0 -k\pi.
    \]
Hence the rotation number of Eq. \x{phi1} is
    \be \lb{rh3}
    \ro_{\phi}(R)= \lim_{k\to +\oo} \f{\phi(\tau_k)-\th_0}{\tau_k}
    = \lim_{k\to +\oo} \f{-k\pi}{k \pi/(-\rho)}=\rho= \rho_R.
    \ee

The final result \x{rhor3} for RSE \x{rse2} then follows immediately from \x{rh1}---\x{rh3}. \qed

    \bb{rmk} \lb{c11}
For the escaping unit circle $\mcc_1$, one has from \x{rhor3} that the critical and the turning shadowing distances are
    \[
    \ul{R}(\mcc_1)=\ol{R}(\mcc_1)=1.
    \]
More generally, for any escaping circle, by the dilation relation of SCs, both of the critical and the turning shadowing distances are equal to the radius of the circle.
    \end{rmk}

Now let us describe the main features for SCs $\mcs_\rth$, where $R>1$ and $\th_0\in\R$.

The shadowing domain is now an annulus
    \be \lb{incl}
    \D_R:= \z\{ r\in \C: R-1\le |r|\le R+1\y\}.
    \ee
   
Note that
    \[
    |r_\rth(t)|^2 =|e^{i t} (1+ R e^{i\phi(t)})|^2 = 1+ R^2 + 2 R \cos \phi(t).
    \]
Since $\phi(t)$ has the range $\R$, one has
    \[
    \min_{t\in \R} |r_\rth(t)|=R-1>0, \andq \max_{t\in \R} |r_\rth(t)|=R+1.
    \]
In particular, any  SC $\mcs_\rth$ is contained in $\D_R$:
    \(
    \mcs_\rth \subset \D_R.
    \)

{\bf Turning points.} By the first equality of \x{rt'}, singular times $t$ of $\mcs_\rth$ such that $r'(t)=0$ are determined by
    \be \lb{Ttk}
    \phi(t) = -k\pi, \qq k\in \Z.
    \ee
By using \x{sec322}, \x{sec34} and \x{Ttk}, $\mcs_\rth$ always admits a bi-sequence $\{\tau_k=\tau_k(\th_0)\}_{k\in \Z}$ of singular times given by
    \be \lb{Tk2}
    \tau_k = -F(\th_0) +F(-k\pi)\equiv \tau_0+ k \pi/|\rho|, \qq k\in \Z.
    \ee
Here these times $\tau_k$ are indexed so that $\tau_k$ is strictly increasing with respect to $k\in \Z$. They satisfy $\tau_k \to \pm \oo$ as $k\to \pm \oo$. As before, singular times give turning points of $\mcs_\rth$
    \be\lb{rtk}
    r_k=r_k(\th_0):=r(\tau_k)
    \equiv (-1)^{k}(R+(-1)^k) e^{i (\tau_0+ k \pi/|\rho|)}.
    \ee
They are located at
    \be \lb{rtk12}
    r_{2k} \in \mcc_{R+1}, \andq r_{2k+1} \in \mcc_{R-1}\qqf k\in \Z,
    \ee
where $\mcc_{R+1}$ and $\mcc_{R-1}$ are the parallel curves of $\mcc_1$, they are
the outer and the inner circles of the annulus $\D_R$ defined by \x{incl}.

Using these notations, the SC $r_\rth(t)$ is evolving as follows. With an initial angle $\th_0$ at hand, SC $r_\rth(t)$ will arrive at a point $r_0=(R+1) e^{i \tau_0}\in \mcc_{R+1}$ on the outer circle. Then

\bu after a time of $\pi/|\rho|$, $r_\rth(t)$ will arrive to the inner circle $\mcc_{R-1}$ at the point $r_1=(R-1) e^{i \tau_1}$, and

\bu after another time of $\pi/|\rho|$, it will return to the outer circle $\mcc_{R+1}$ at the point $r_2=(R+1) e^{i \tau_2}$, and so on.

{\bf Periodicity and quasi-periodicity.} The types of solutions $\phi(t)=\phi_\rth(t)$ and SCs $r(t)=r_\rth(t)$ depend on whether $\rho=\rho_R$ is rational. By \x{r1t}, one has
    \be \lb{sc-Rd1}
    r(t) = e^{i t}\bigl(1+ R e^{i\phi(t)}\bigr):= e^{i t} \tl r(t),
    \ee
where $\tl r(t)$ satisfies
    \be \lb{tlr}
    \tl r\z(t+2\pi/|\rho|\y)=1+ R e^{i \phi(t+2\pi/|\rho|)}=1+ R e^{i(\phi(t)-2\pi)}\equiv \tl r(t).
    \ee
See \x{ph-p}. In fact, $\tl r(t)$ has the minimal period $2\pi/|\rho|$.

Using formulas \x{vr} and \x{rhor3} for rotation numbers, we can distinguish the following two cases for $R>1$.

{\bf Case 1: Subharmonic SCs.} There are co-prime integers $p>q\ge 1$ such that
    \be \lb{Rpq}
    R=\rpq := \f{p}{\sqrt{p^2-q^2}}>1,
    \ee
i.e.
    \be \lb{Rpq1}
    \rho_\rpq 
    =-\f{q}{p}\in(-1,0), \andq \ro(\rpq) = \f{p-q}{p}\in(0,1).
    \ee
In this case, it follows from \x{sc-Rd1} and \x{tlr} that $e^{i t}$ and $\tl r(t)$ have the minimal periods $2\pi$ and $2p\pi/q$ respectively. Hence $r_{\rpq,\th_0}(t)$ is periodic of the minimal period $T_{\min}=2p\pi$. Thus all $r_{\rpq, \th_0}(t)$, $\th_0\in \R$ are subharmonic SCs of the minimal period $2p\pi$.

Further properties on the shapes of these SCs are as follows.

    \bb{thm} \lb{Rdy1a}
Let $R=\rpq>1$ be as in \x{Rpq} and \x{Rpq1}. For any $\th_0$, $r_{\rpq,\th_0}(t)$ is a subharmonic SC of the minimal period $T_{\min}=2p \pi$. Moreover, it admits precisely $2q$ turning points $\z\{r_{k}(\th_0)\y\}_{k=0,1,\dd,2q-1}$, half of which are on the outer circle $\mcc_{\rpq+1}$ and another half on the inner circle $\mcc_{\rpq-1}$.
    \end{thm}

\Proof We have known that $r_{\rpq,\th_0}(t)$ is a subharmonic SC of the minimal period $2p \pi$. Since $p$ and $q$ are co-prime, one sees from \x{Tk2}---\x{rtk12} that
    \[\bb{split}
    r_{2k} & = +(\rpq+1) e^{i (\tau_0+ 2 k p\pi/q)}\in \mcc_{\rpq+1},\\
    r_{2k+1} & = -(\rpq-1) e^{i (\tau_0+ (2 k+1) p \pi/q)}\in \mcc_{\rpq-1},
    \end{split} \qq k\in \Z.
    \]
These have given precisely $q$ turning points on $\mcc_{\rpq+1}$ and another $q$ turning points on $\mcc_{\rpq-1}$. \qed

Due to the properties described in the theorem, these SCs $r_{\rpq,\th_0}(t)$ are refereed as the {\it $q/p$-subharmonic SCs}.
They look like color clouds or flowers.
See Figures \ref{sc-sub1} and \ref{sc-sub2}.

\begin{figure}[ht]
\centering
\includegraphics[width=6.5cm, height=3.7cm]{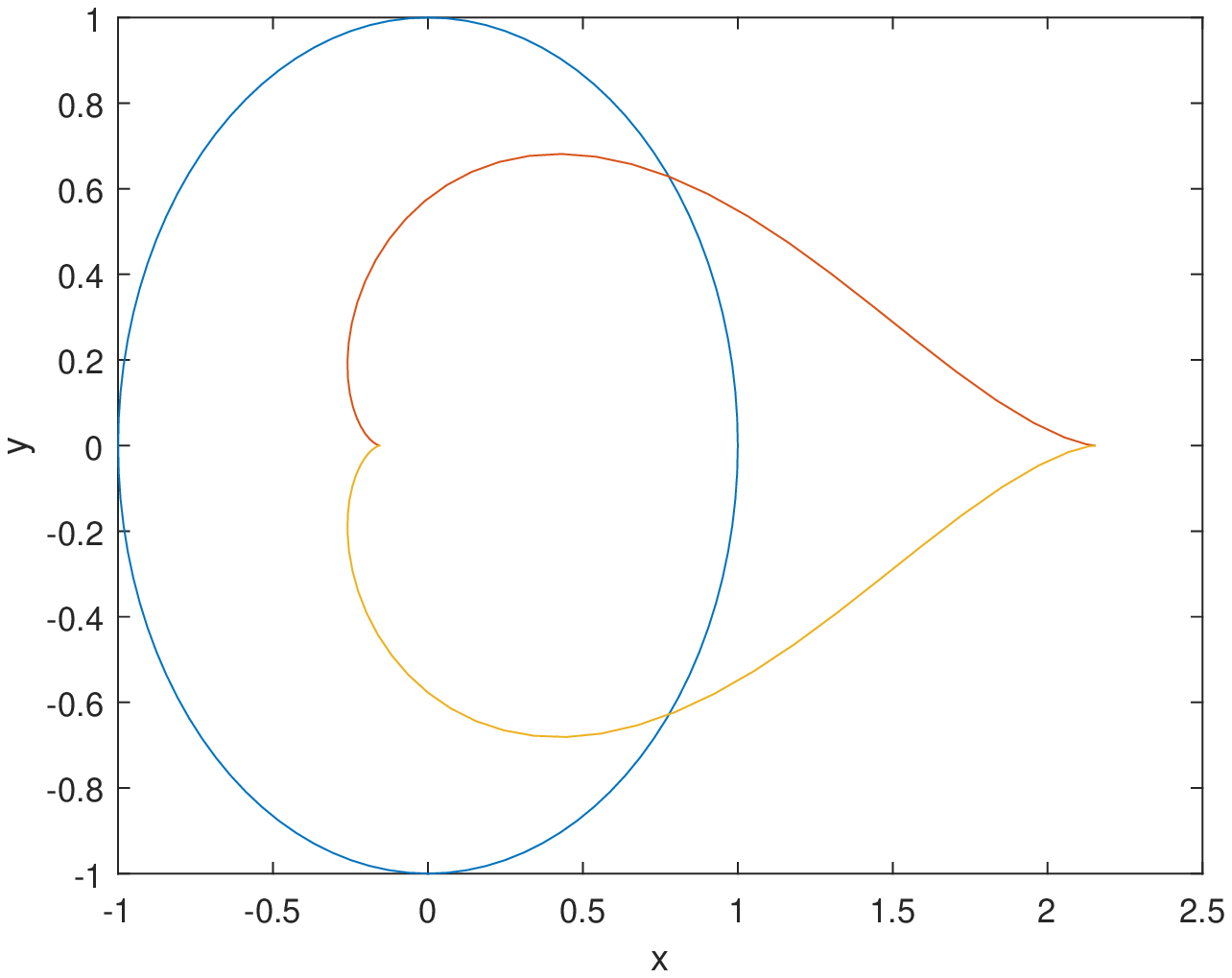}\hspace{1cm}
\includegraphics[width=6.5cm, height=3.7cm]{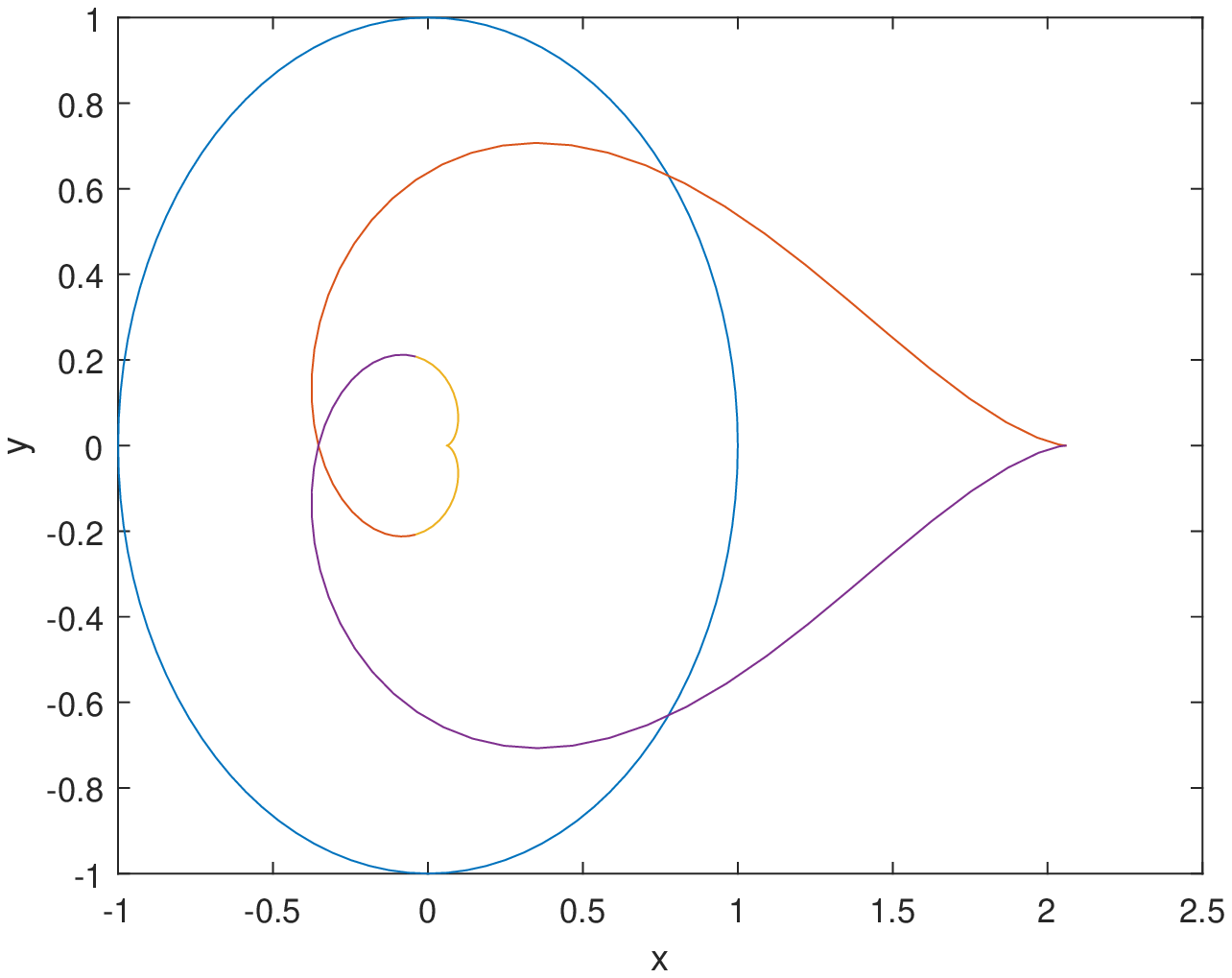}\\
\includegraphics[width=7cm, height=2.4cm]{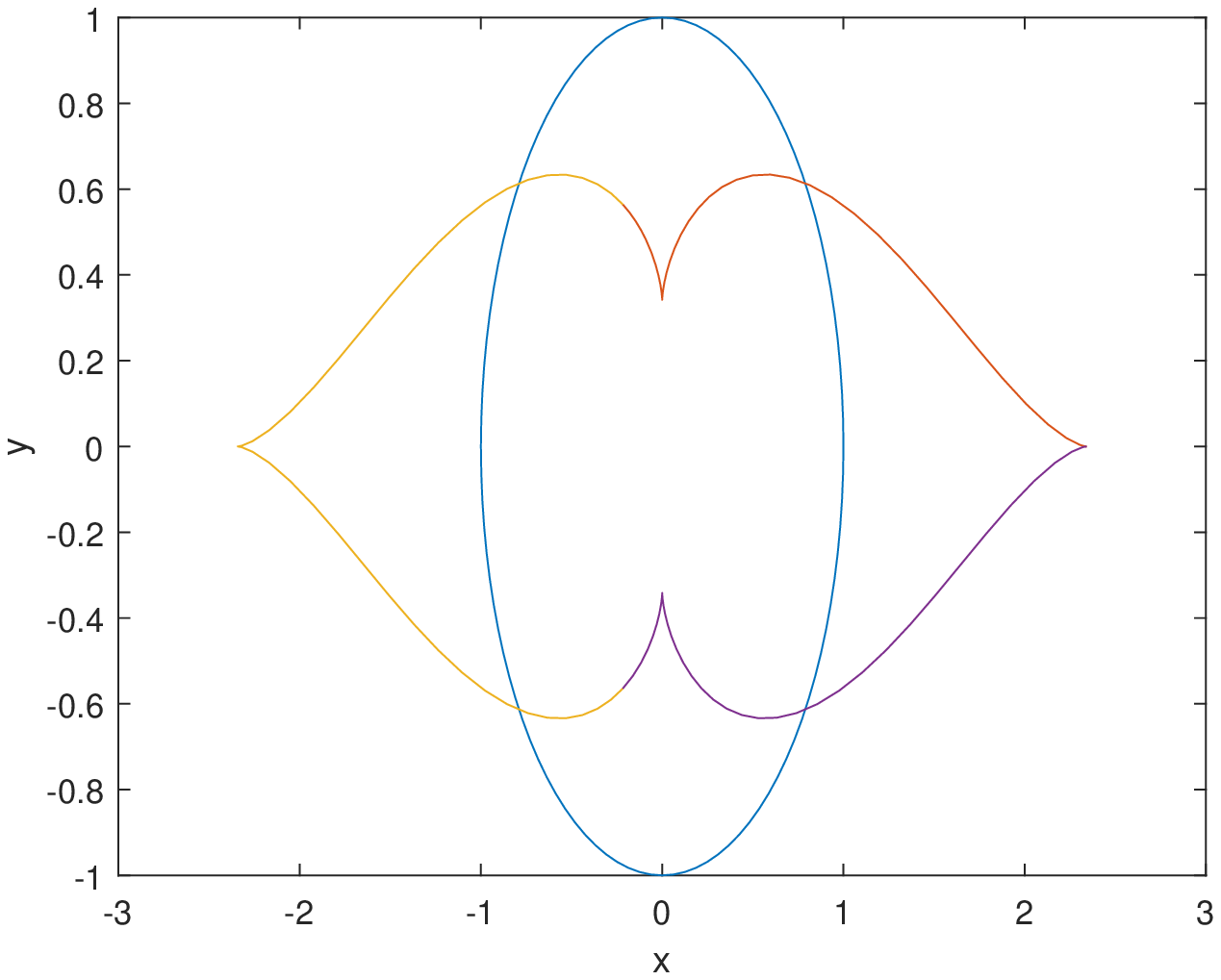}\hspace{1cm}
\includegraphics[width=6.5cm, height=3.2cm]{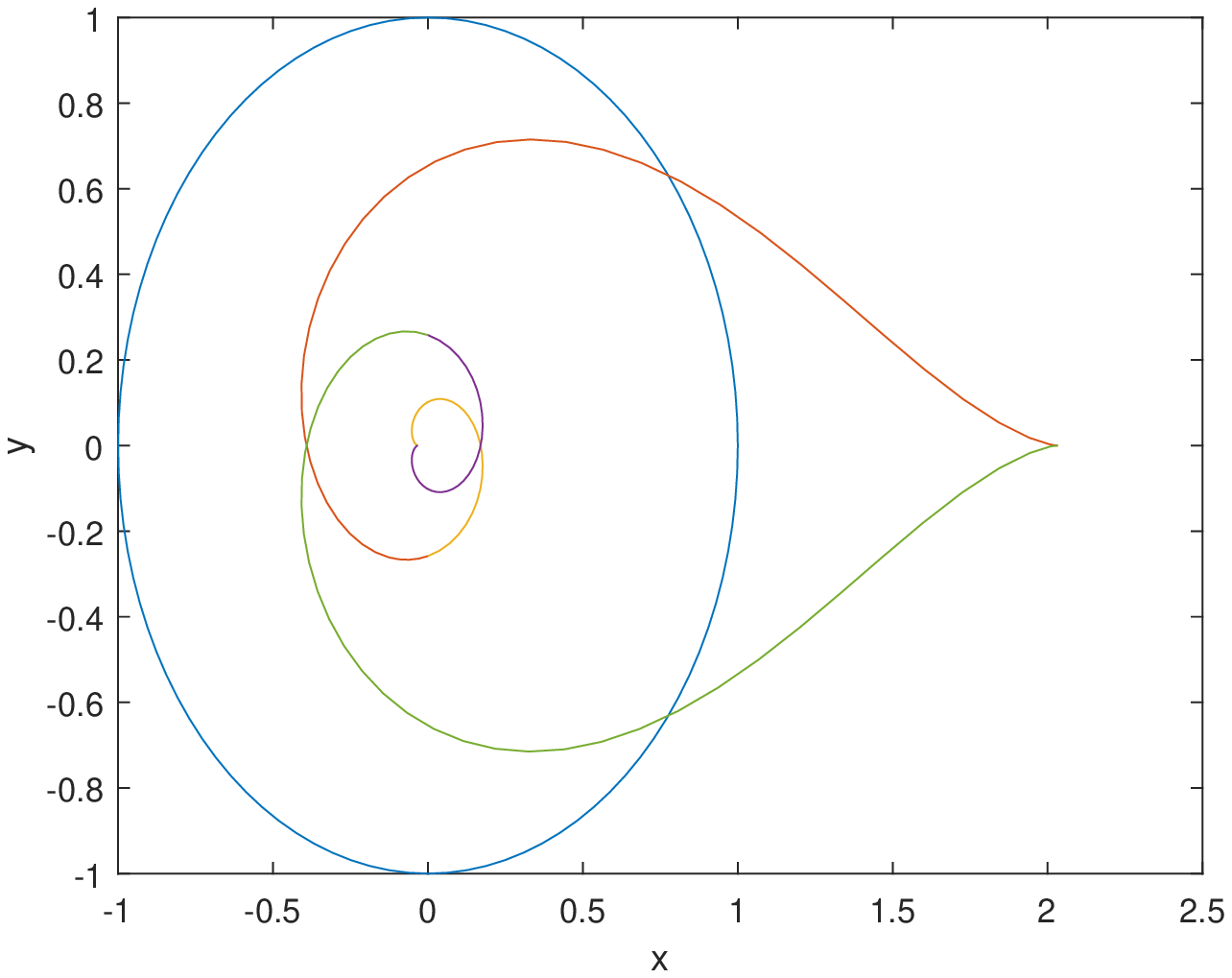}
\caption{Subharmonic SCs $\mcs_{R,0}$ to $\mcc_1$.
Upper-Left: $R=R_{1/2}$, $T_{\min}=4\pi$; Upper-Right: $R=R_{1/3}$, $T_{\min}=6\pi$.
Lower-Left: $R=R_{2/3}$, $T_{\min}=6\pi$;  Lower-Right: $R=R_{1/4}$, $T_{\min}=8\pi$.
}
\label{sc-sub1} 
\end{figure}

\begin{figure}[ht]
\centering
\includegraphics[width=6.5cm, height=6cm]{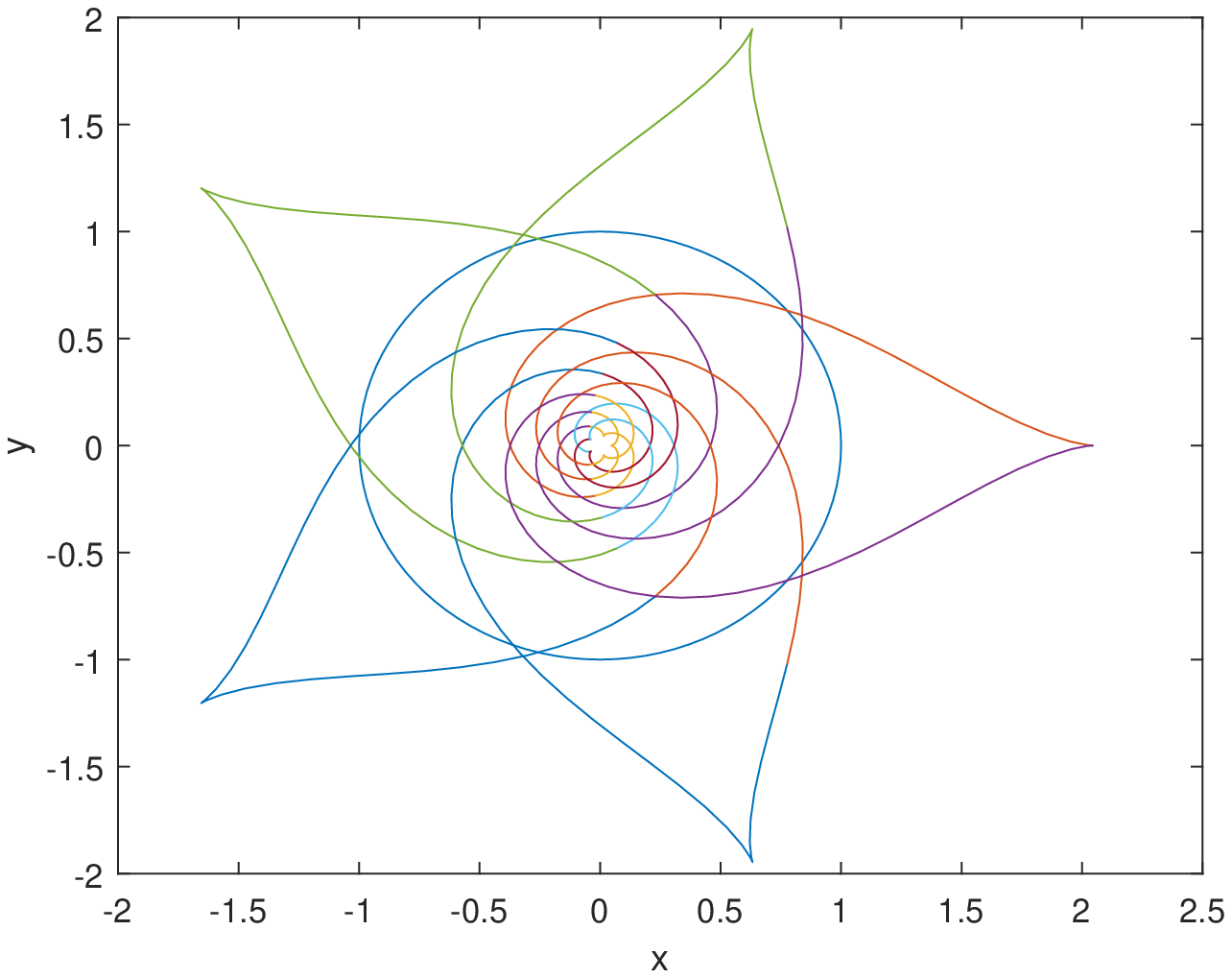}\hspace{1cm}
\includegraphics[width=6.5cm, height=6cm]{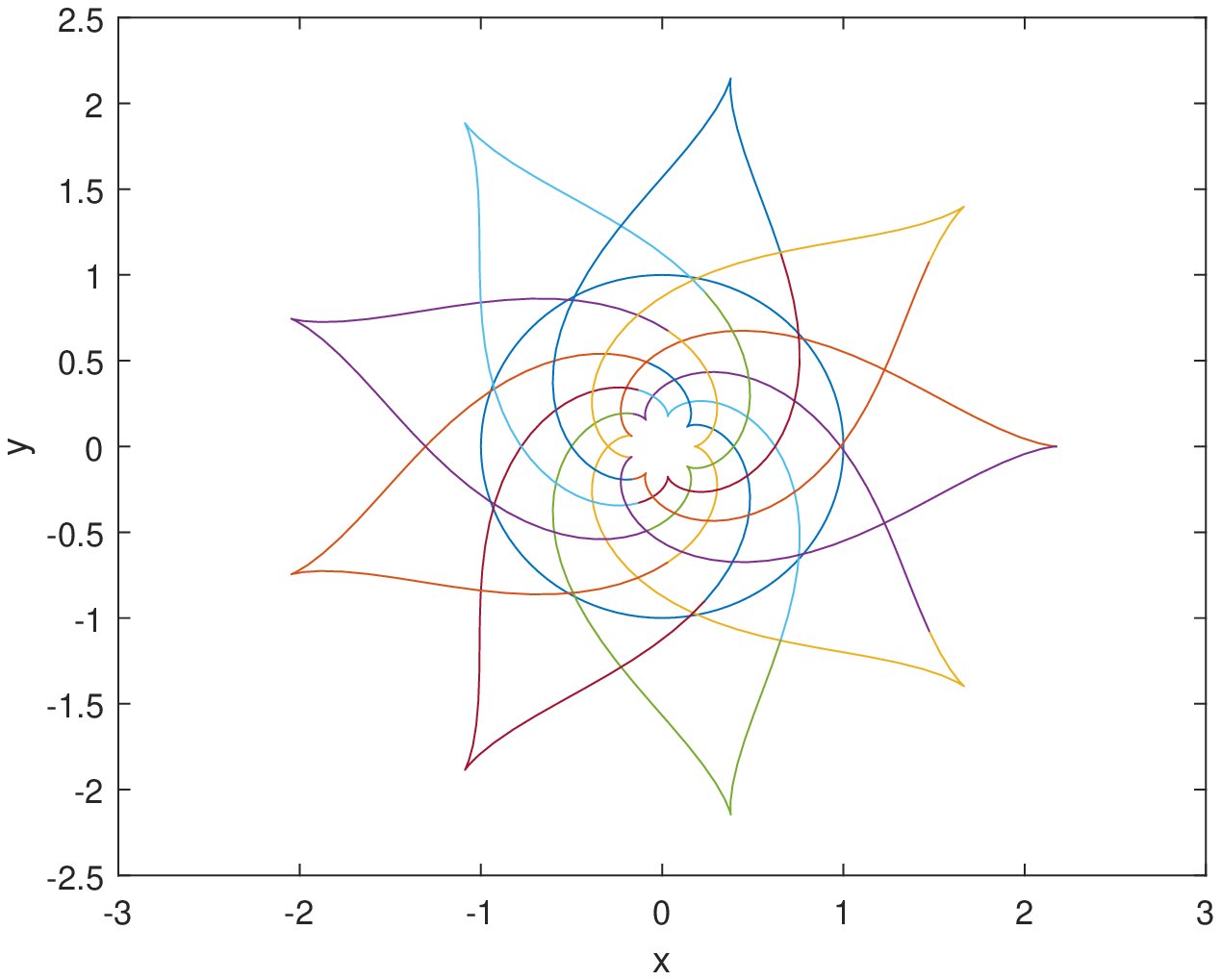}\\
\caption{Subharmonic SCs $\mcs_{R,0}$ to $\mcc_1$ with the minimal period $T_{\min}=34\pi$.
Left: $R=R_{5/17}$; Right: $R=R_{9/17}$. }
\label{sc-sub2} 
\end{figure}

{\bf Case 2: Ergodic SCs.} $R>1$ satisfies
    \be \lb{Rpqn}
    R\ne \rpq \ \ \forall p>q\ge 1,
    \ee
i.e. $\rho_R$ and $\ro(R)$ are irrational. By \x{sc-Rd1} and \x{tlr} again, all SCs $r_\rth(t)$ are quasi-periodic SCs of SE \x{rse1}.

    \bb{thm} \lb{Rdy1b}
Let $R>1$ be as in \x{Rpqn}. Then, for any  $\th_0\in \R$, one has

\bu SC $r_\rth(t)$ always admits an infinite sequence of turning points $\z\{r_{2k}(\th_0)\y\}_{k\in \Z}$ on the outer circle $\mcc_{R+1}$ and another infinite sequence of turning points $\z\{r_{2k+1}(\th_0)\y\}_{k\in \Z}$ on the inner circle $\mcc_{R-1}$.

\bu SC $r_\rth(t)$ is a quasi-periodic SC and 
is dense in the shadowing domain $\D_R$:
    \be \lb{erg1}
    \ol{\z\{r_\rth(t): t\in \R \y\}} = \D_R.
    \ee
    \end{thm}

\Proof As $\ro=\ro(R)$ is irrational, it follows from \x{rtk} that all turning points $r_k(\th_0), \ k\in \Z$ are different. Hence $r_\rth(t)$ has always infinitely many turning points which are located on the boundary circles $\mcc_{R\pm 1}$ of $\D_R$.
The shadowing domain in  \x{Dr} of $r_\rth(t)$ is currently
    \[
    \bigl\{e^{i t} + R e^{i \psi}: t\in\R, \, \psi\in \R \bigr\}.
    \]
It is the annulus $\D_R$ defined in \x{incl}. Hence the density result \x{erg1} follows from Theorem \ref{main2}.
\qed

Due to the density result \x{erg1} as in the theorem, these quasi-periodic SCs for $R$ as in \x{Rpqn} are called  {\it ergodic SCs}. For some typical ergodic SCs, see Figures \ref{Erg3} and \ref{Erg4}.

\begin{figure}[ht]
\centering
\includegraphics[width=6.5cm, height=5.2cm]{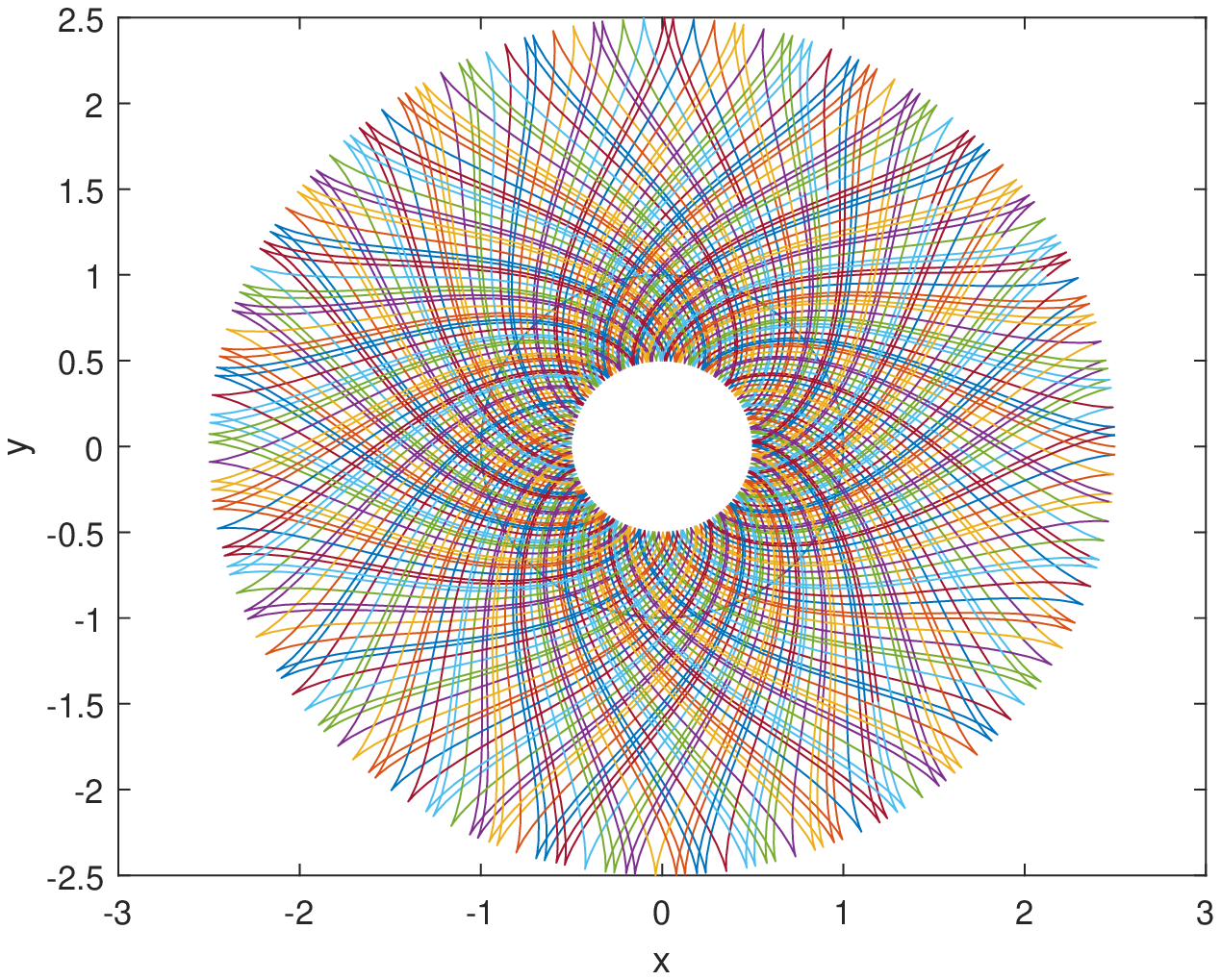}\hspace{1cm}
\includegraphics[width=6.5cm, height=5.2cm]{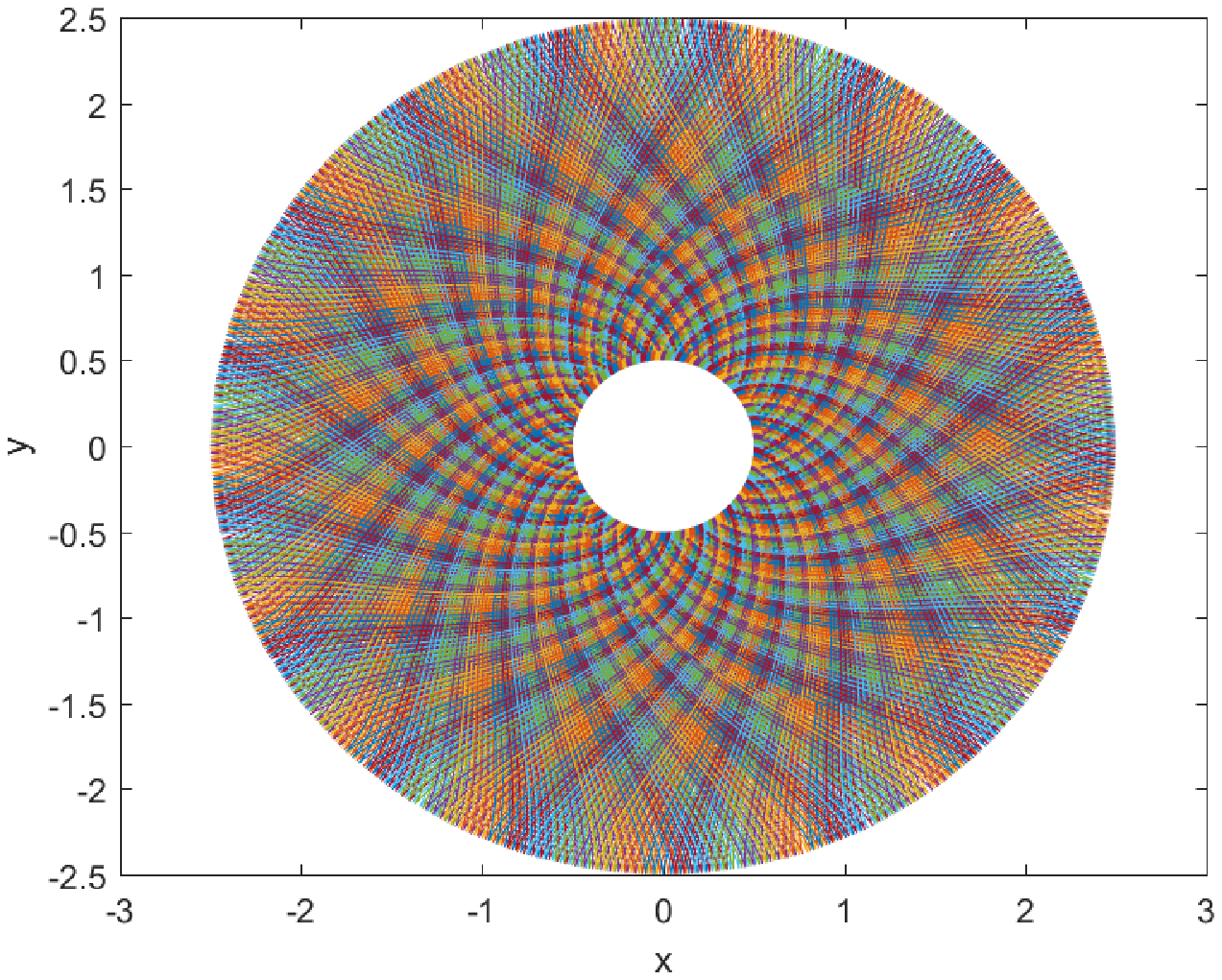}
\caption{Ergodic $\mcs_\rth$ with shadowing distance $R=3/2$ and initial angle $\th_0=0$.
The time spans are $[0,500\pi]$ (left) and $[0,2500\pi]$ (right), respectively.}
\label{Erg3} 
\end{figure}

\begin{figure}[ht]
\centering
\includegraphics[width=6.5cm, height=6.2cm]{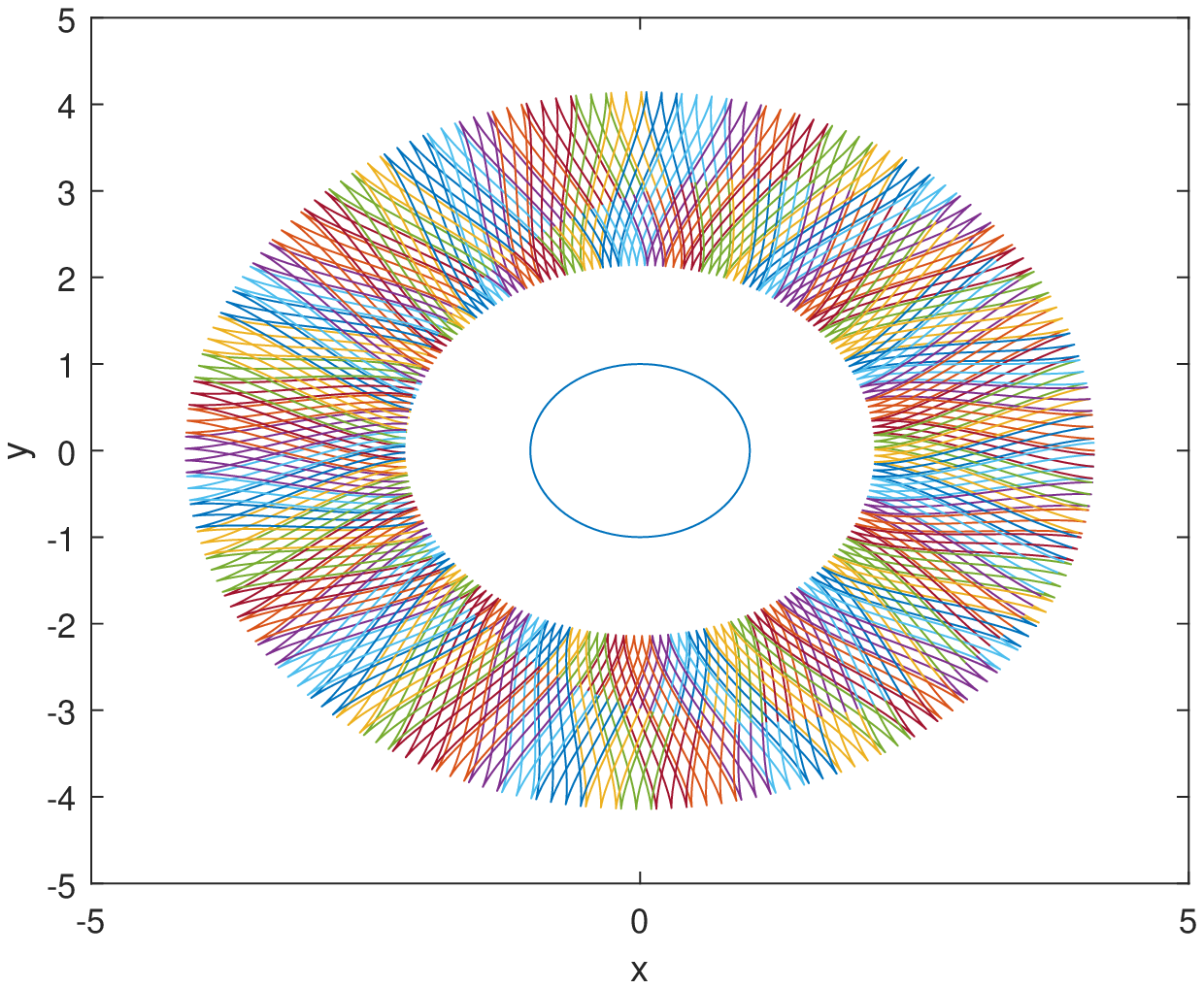}\hspace{1cm}
\includegraphics[width=6.5cm, height=6.2cm]{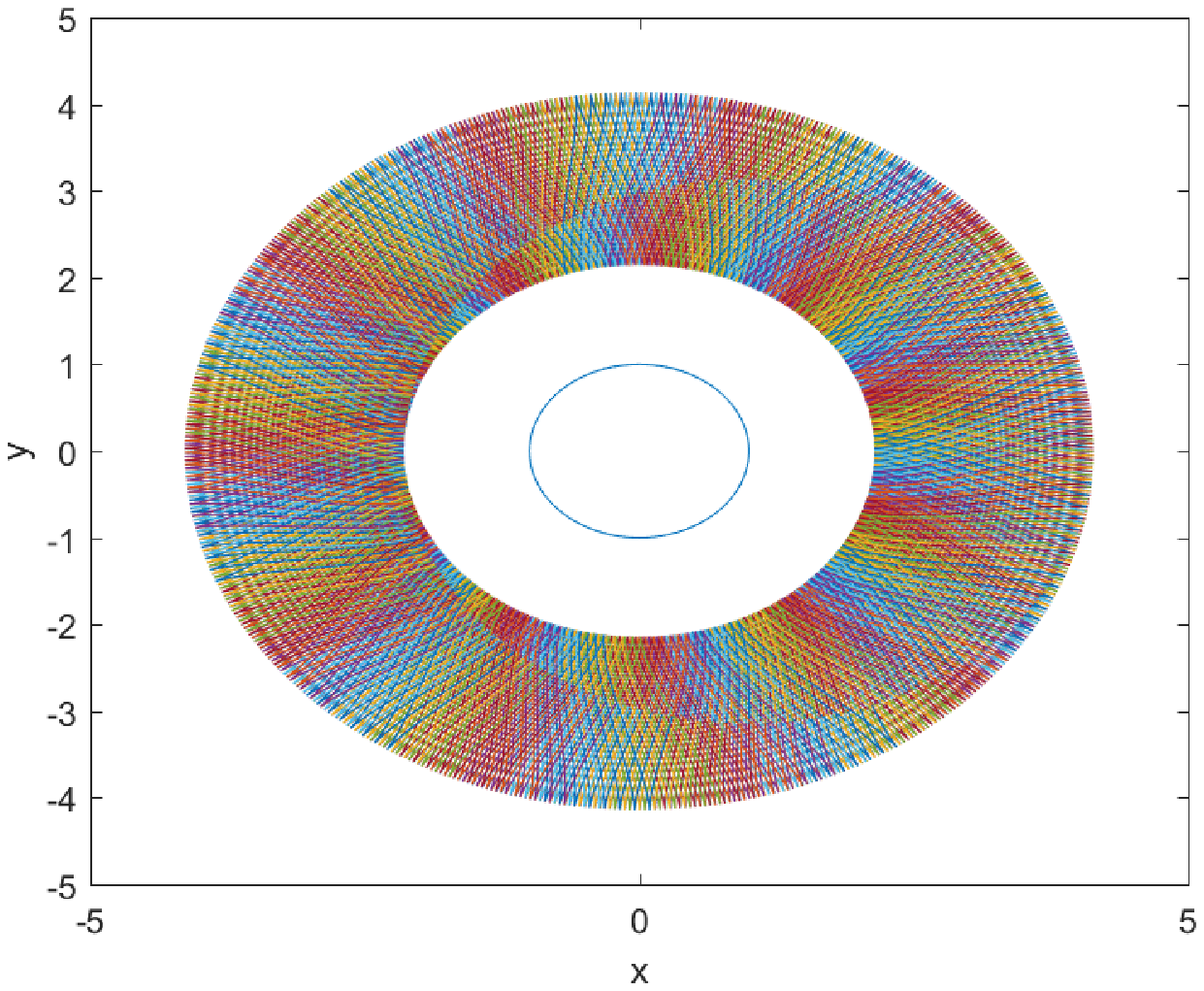}
\caption{Ergodic $\mcs_\rth$ with shadowing distance $R=\pi$ and initial angle $\th_0=0$.
The time spans are $[0,500\pi]$ (left) and $[0,4000\pi]$ (right), respectively.}
\label{Erg4} 
\end{figure}

\subsection{SCs to ellipses} \lb{sc-ell}

Let us choose the escaping curve as the following ellipse
    \[
    \E_b: \q \vec r=r_{0;b}(t) := (\cos t, b \sin t).
    \]
Here $b>0$. In the Descartes coordinates, it is
    \[
    \E_b: \q x^2 +\f{y^2}{b^2} =1.
    \]
The RSE \x{RSE} to $\E_b$ is
    \be\lb{rseb}
    \th'=  -\f{1}{R}(\sin t\sin \th +b \cos t\cos \th), \qq \th(0)=\th_0,
    \ee
whose rotation number is denoted by
    \[
    \ro=\ro_b(R), \qq R>0, \ b>0.
    \]
It is a continuous function of $(R,b)\in(0,+\oo)^2$. We can obtain some lower and upper bounds for the critical shadowing distance to $\E_b$.

    \bb{lem} \lb{rn-p}
For the ellipse $\E_b$, there hold
    \be \lb{Rb}
    \min\{b,1\} \le \ul{R}_b :=\ul{R}(\E_b)\le \f{\ell(\E_b)}{2\pi}=\f{1}{\pi}\z( b {\mathbb E}(1-1/b^2) + {\mathbb E}(1-b^2)\y) \qqf b>0.
    \ee
    \ifl
Moreover, one has
    \be \lb{Rb01}
\ro_b(R) \equiv 1\mbox{ for } R\in(0,\ul{R}_b], \andq \lim_{R\to +\oo} \ro_b(R) =0.
    \ee
    \fi
    \end{lem}

\Proof The upper bound in \x{Rb} can be obtained from \x{ulr-u1}, because $\E'_b$ has the rotation index $1$. Here $\ell(\E_b)$ is just the usual perimeter of ellipses, given by the elliptic functions of $b$.

To obtain the lower bound in \x{Rb}, we argue as in the proof of Theorem \ref{rho-01}. As in \x{ths}, we simply use the transformation
    \(
    \phi(t): = \th(t)-t
    \)
for RSE \x{rseb}. Then $\phi(t)$ satisfies the time periodic ODE
    \be \lb{rseb1}
    \f{\rd\phi}{\dt} =- \f{1}{R} \z(\sin t\sin (t+\phi) + b \cos t\cos(t+\phi) \y)-1 =: g(t,\phi).
    \ee
Note that
    \[
    g(t,0)= -\f{1}{R}(\sin^2 t+b \cos^2 t)- 1<0
    \]
for all $t$. On the other hand, let $R< \min\{b,1\}$. Then, for any $t$,
    \beaa
    g(t,-\pi)\EQ \f{1}{R}(\sin^2 t+b \cos^2 t)- 1\ge \f{1}{R} \min\{b,1\} - 1>0.
    \eeaa
By Lemma \ref{existenceper}, Eq. \x{rseb1} admits $2\pi$-periodic solutions $\phi=\phi_\pm(t)$. Similar to the proof of Theorem \ref{rho-01}, the rotation numbers of Eq. \x{rseb1} and Eq. \x{rseb} are respectively $0$ and $1$ if $R< \min\{b,1\}$. This yields the desired lower bound in \x{Rb}.
\qed

For $b=1$, both bounds in \x{Rb} are precisely equal to $\ul{R}(\mcc_1)=1$, the critical shadowing distance of $\mcc_1$. See Remark \ref{c11}. In Figure \ref{rnb2}, we have plotted the function of rotation numbers $\ro_2(R)$ of $R$ by choosing $b=2$. Like the case of circles, it has a platform and then is strictly decreasing in $R$. One sees for the ellipse $\E_2$ that the critical and the turning shadowing distances $\ul{R}(\E_2)$ and $\ol{R}(\E_2)$ are equal and  are approximately $1.445$. Notice that the constant $\mu=\ell_0/2\pi$ is approximately 1.542. This is consistent with the result we have given in Theorem \ref{circ}.

\begin{figure}[ht]
\centering
\includegraphics[width=8cm, height=5cm]{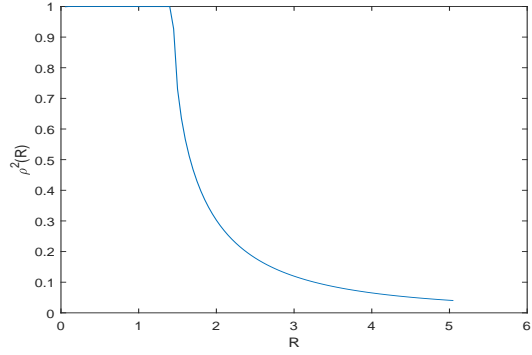}
\caption{Rotation numbers $\ro_b(R)$, $b=2$ of RSE \x{rseb}, as a function of $R$. }
\label{rnb2} 
\end{figure}

Some numerical observations to the types of SCs to the ellipse $\E_b$ with the choice of $b=2$ are as follows.

{\bf Periodic SCs.} When $R$ is relatively small, say $R=1$, we have the $2\pi$-periodic SCs $r_{R,\th_{R,b}^\pm;b}(t)$ and those SCs which are approaching to the periodic SCs. See Figure \ref{sc-ell10}. These are similar to the case of the circle EC. However, different from the circle case, general SCs are approaching to different $2\pi$-periodic closed SCs as time evolves to $+\oo$ or to $-\oo$. That is, the $2\pi$-periodic SCs $r_{R,\th_R^\pm}(t)$ to $\mcc_1$, which have same trajectories, have split into different $2\pi$-periodic SCs $r_{R,\th_{R,b}^\pm;b}(t)$ to ellipses.

\begin{figure}[ht]
\centering
\includegraphics[width=7.5cm, height=5cm]{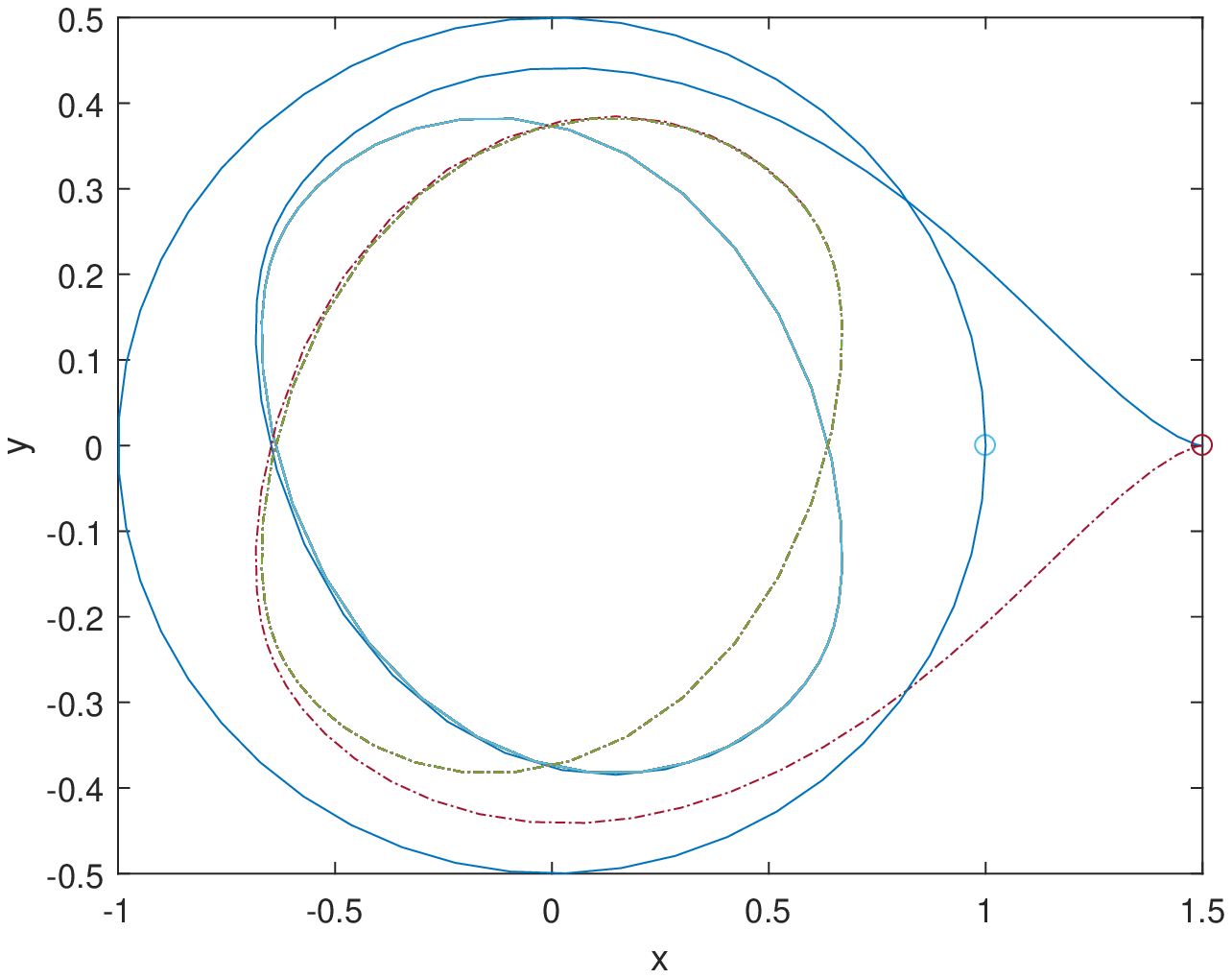}\hspace{1cm}
\includegraphics[width=6cm, height=6.5cm]{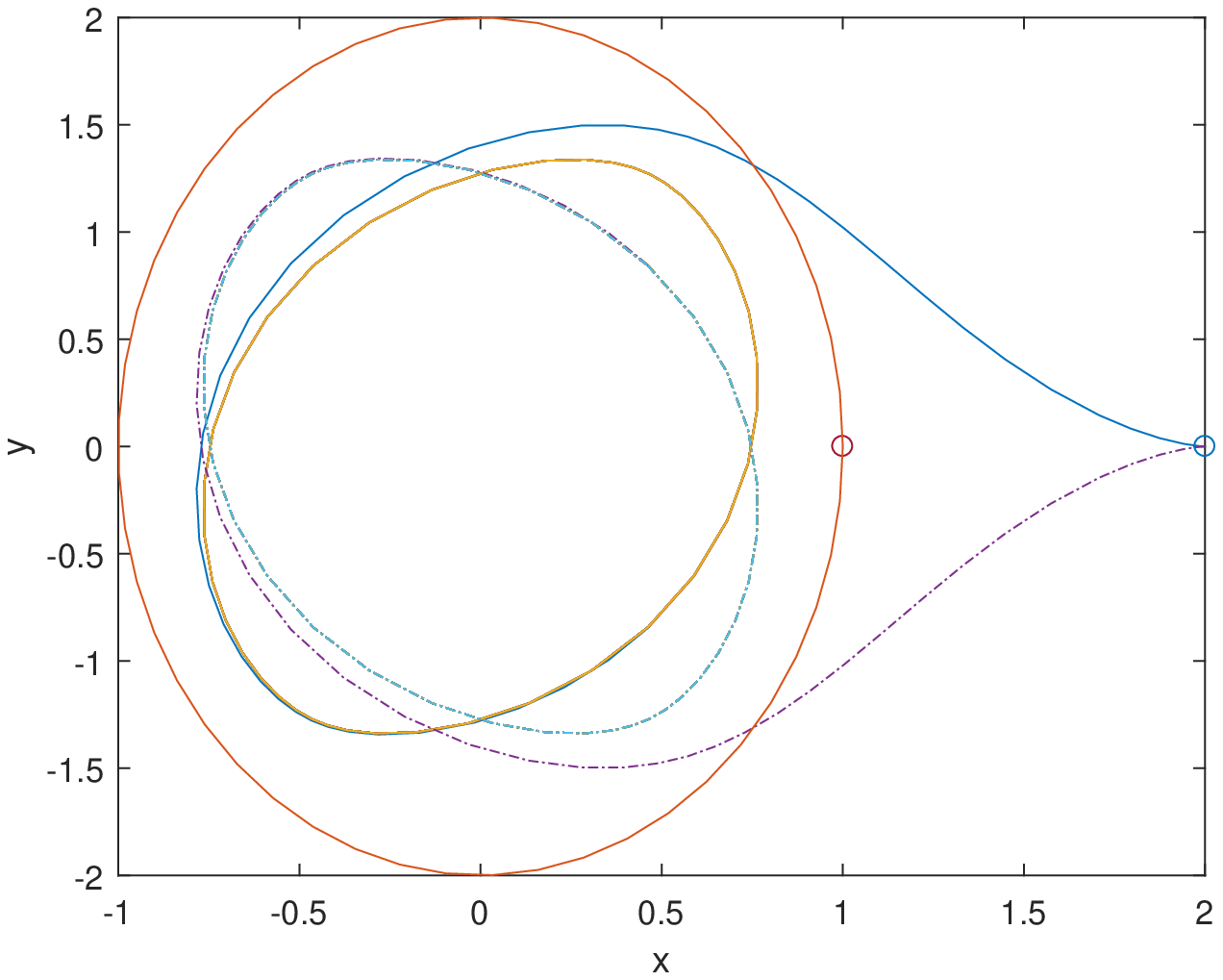}
\caption{SCs $S_{R,0;b}$ to ellipses $\E_b$ with small shadowing distances $R$, where $(R,b)=(1/2,1/2)$ (left) and $(R,b)=(1,2)$ (right).
}
\label{sc-ell10} 
\end{figure}

\ifl
\begin{figure}[ht]
\centering
\includegraphics[width=6.5cm, height=5cm]{scb2a}\hspace{1cm}
\includegraphics[width=6.5cm, height=8cm]{scb2b}\\
\caption{SCs $S_{R,\th_0;b}$ to ellipse $\E_b$ with shadowing distances $R<\ul{R}_b$, where $(R,\th_0,b)=(1.4,0,2)$ (left) and $(R,\th_0,b)=(1.4,\pi/2,2)$ (right). } 
\label{sc-ellr14} 
\end{figure}

One sees from Figures \ref{sc-ell10} and \ref{sc-ellr14} that, when $t\to +\oo$ and $t\to-\oo$, each SC to ellipse is approaching to different closed SCs. Comparing with Figure \ref{sc-circle5} for the circle EC, the limiting circle is now split into two closed SCs.
\fi

{\bf Ergodic SCs.} When $R$ is relatively large, the irrationality of $\ro(R)$ will lead to ergodic SCs. 
See Figure \ref{scell250b}.

\ifl
\begin{figure}[ht]
\centering
\includegraphics[width=5.5cm, height=5cm]{scbec1}\hspace{1cm}
\includegraphics[width=5.5cm, height=5cm]{scbec2}
\caption{Ergodic SC $S_{2,0;2}$ to the ellipse with a large shadowing distance $R$.
Here the time spans are $[0,20\pi]$ (left) and $[0,2000\pi]$ (right).}
\label{sc-ell2} 
\end{figure}
\fi

\begin{figure}[ht]
\centering
\includegraphics[width=5cm, height=5cm]{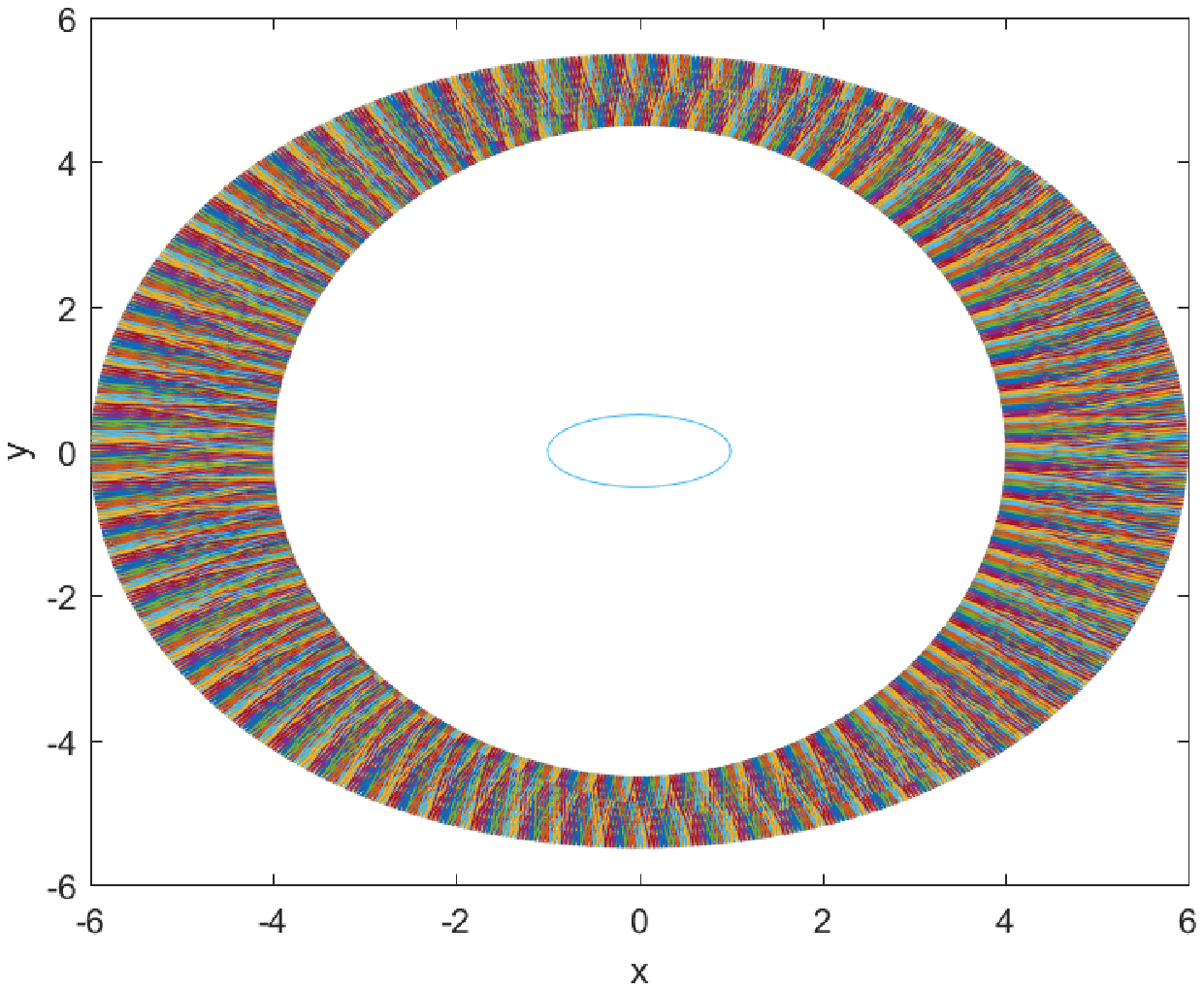}\hspace{1.5cm}
\includegraphics[width=4.5cm, height=6.5cm]{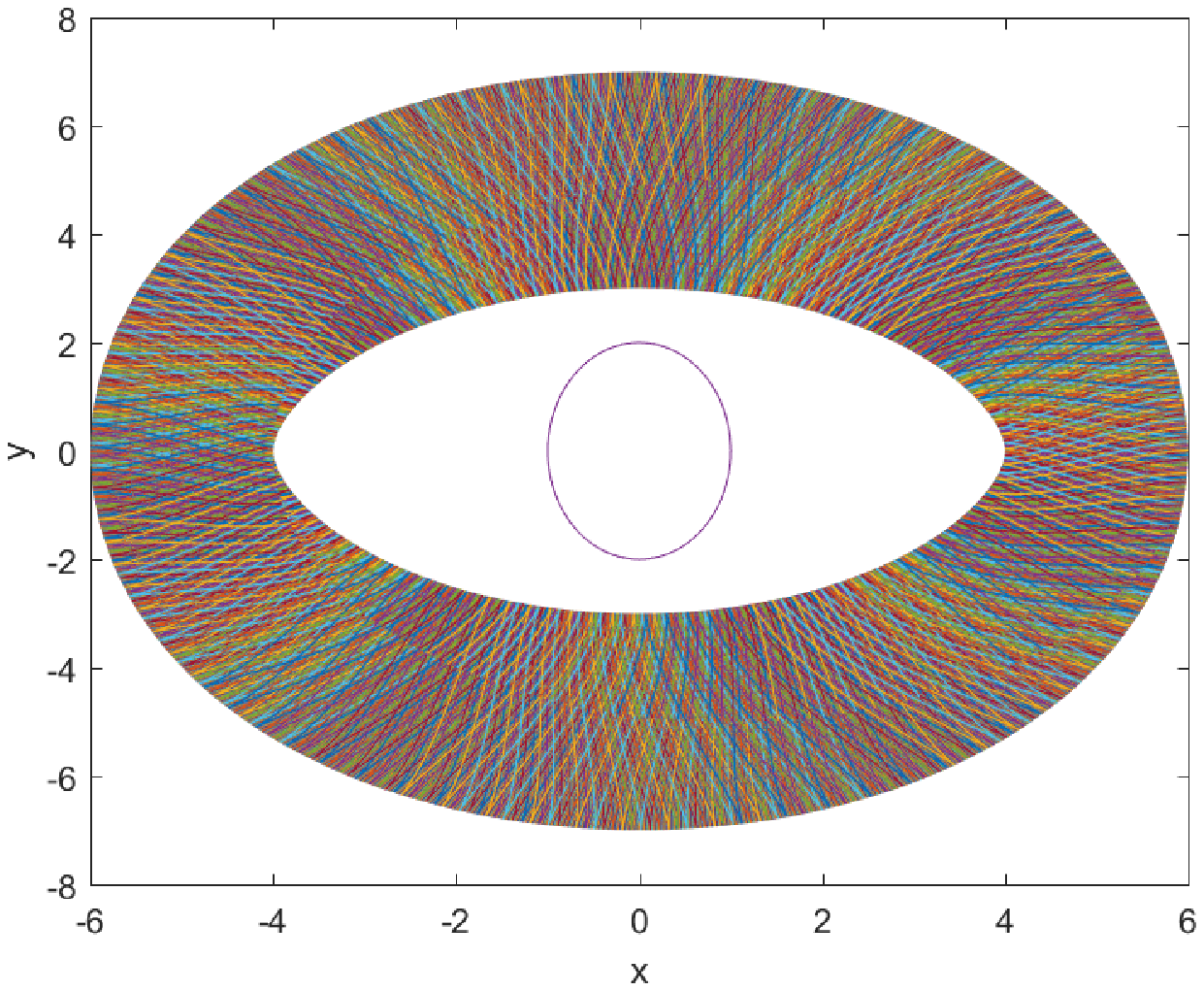}
\caption{Ergodic SCs $S_{5,0;1/2}$ (left) and $S_{5,0;2}$ (right) to ellipses with a large distance $R$. }
\label{scell250b} 
\end{figure}

{\bf Subharmonic SCs.} When $R$ is larger than the turning shadowing distance $\ol{R}(\E_b)$ with rational rotation number $\ro_b(R)$, we have predicted in Theorem \ref{main1} the existence of subharmonic SCs. However, as these subharmonic SCs are sensitive in shadowing distance, it is not easy to precisely simulate subharmonic SCs.


\ifl
\ysb{See Figure \ref{subharm}.

\begin{figure}[ht]
\centering
\includegraphics[width=4.5cm, height=4cm]{scb2s21}\hspace{0.5cm}
\includegraphics[width=4.5cm, height=4cm]{scb2s31}\hspace{0.5cm}
\includegraphics[width=4.5cm, height=4cm]{scb2s41}
\caption{Subharmonic SCs to the ellipse $\E_2$, where $q=1$ and $p=2,3,4$ (from left to right).  }
\label{subharm} 
\end{figure}

\fi

\section{Conclusion and a Conjecture}
\setcounter{equation}{0} \lb{fifth}

\subsection{Conclusion} \lb{concl}

In this paper we have introduced a dynamical system model called the shadowing problem. When the shadower is shadowing  the escaper in the Euclidean spaces, we have derived the shadowing equations and the extended shadowing equations. When escaping curves are chosen as planar closed curves, we have derived the reduced shadowing equations for planar shadowing curves. The complete structure and types of planar shadowing curves can then be determined using the rotation numbers of the circle dynamics, including the famous Denjoy theorem. Moreover, even when the escaping curve is the circle or ellipses, we found that the shadowing problem admits many interesting non-trivial shadowing curves.

\subsection{A conjecture} \lb{conj}

We have only given in this paper a beginning study to the shadowing problems using dynamical systems theory.
Recall that for a general regular planar closed curve $\E$, considered as an escaping curve, we have introduced two important notions --- the critical shadowing distance $\ul{R}(\E)$ and the turning shadowing distance $\ol{R}(\E)$. Their roles in characterizing shadowing curves are displayed in Theorem \ref{main1} and Theorem \ref{main2}. From Lemma \ref{upbound}, Theorem \ref{rho-01} and Theorem \ref{mono}, it is interesting that these are closely related with the $1$-dimensional geometrical quantities of $\E$, like the perimeter $\ell(\E)$, the area $\A(\E)$, and the rotation index $\om(\E)$. On the other hand, by observing the simple examples like the circles and ellipses, it seems that these quantities $\ul{R}(\E)$ and $\ol{R}(\E)$ are coincident. We conjecture that this is true for more general class of closed curves.

    \bb{conj} \lb{conj1}
For any smooth strictly convex closed curve $\E$ on the plane, there holds
    \[
    \ul{R}(\E)= \ol{R}(\E).
    \]
    \end{conj}

Once Conjecture \ref{conj1} is true, $\ul{R}(\E)= \ol{R}(\E)$ can be understood as a geometrical quantity  which cuts the $1$-dimensional and the $2$-dimensional geometries of $\E$, from a point view of dynamical systems.

Finally, it is also of interest to study shadowing problems in higher dimensional Euclidean spaces or in other spaces.

\end{document}